\documentclass{amsart}[10pt]
\usepackage{mathrsfs}
\usepackage{amssymb}
\usepackage{amsfonts}
\usepackage{amsmath,color,esvect,tikz}
\usepackage[colorlinks,linkcolor=blue,urlcolor=cyan,citecolor=blue]{hyperref} 
\usepackage{genyoungtabtikz}
\usepackage[section]{placeins}

\usepackage{fullpage}
\usepackage[textwidth=0.85in, textsize=tiny]{todonotes}
\usepackage[backend=biber,style=alphabetic,sorting=nyt]{biblatex}
\newcommand{\hobox}[3]{\draw (0+#1,0-#2) rectangle (1+#1,-1-#2)++(-0.5,+0.5) node {$ #3$};}

\newcommand\reduline{\bgroup\markoverwith
      {\textcolor{red}{\rule[0.5ex]{2pt}{1pt}}}\ULon}

\DeclareUnicodeCharacter{0306}{}

\newcommand{\al}{\alpha}

\newcommand{\lam}{\lambda}

\newcommand{\ep}{\varepsilon}

\newcommand{\frb}{\mathfrak{b}}
\newcommand{\fre}{\mathfrak{e}}

\newcommand{\frg}{\mathfrak{g}}
\newcommand{\frh}{\mathfrak{h}}

\newcommand{\frp}{\mathfrak{p}}

\newcommand{\fru}{\mathfrak{u}}

\newcommand{\frso}{\mathfrak{so}}
\newcommand{\frsp}{\mathfrak{sp}}

\newcommand{\caO}{\mathcal{O}}

\newcommand{\caV}{\mathcal{V}}
\newcommand{\caW}{\mathcal{W}}

\newcommand{\C}{\mathbb C}

\newcommand{\mc}[1]{\mathcal{#1}}

\newcommand{\mf}[1]{\mathfrak{#1}}
\renewcommand{\subset}{\subseteq}
\newcommand{\rar}{\rightarrow}

\newcommand{\hs}{ \mathfrak{h}^*}

\newcommand{\Ann}{\operatorname{Ann}}

\newcommand{\nodd}{}

\newcommand{\ev}{\mathrm{ev}}
\newcommand{\od}{\mathrm{od}}

\newtheorem{Thm}{Theorem}[section]
\newtheorem{Cor}[Thm]{Corollary}
\newtheorem{Pro}[Thm]{Proposition}
\newtheorem{Lem}[Thm]{Lemma}

\theoremstyle{definition}
\newtheorem{Rem}[Thm]{Remark}
\newtheorem{definition}[Thm]{Definition}

\newtheorem{example}[Thm]{Example}

\numberwithin{equation}{section}

\begin{document}

\title[Cells and associated varieties ]{{On the cells and  associated varieties of highest weight Harish-Chandra modules}}

\author{ Zhanqiang Bai, Yixin Bao, Zhao Liang and Xun Xie}

\address[Bai]{School of Mathematical Sciences, Soochow University, Suzhou 215006, China}
\email{zqbai@suda.edu.cn}

\address[Bao]{School of Sciences, Harbin Institute of Technology, Shenzhen, 518055,   China}
\email{baoyixin@hit.edu.cn }
\address[Liang]{School of Mathematical Sciences, Soochow University, Suzhou 215006,  China}
\email{20224207008@stu.suda.edu.cn}

\address[Xie]{School of Mathematics and Statistics, Beijing
Institute of Technology, Beijing 100081, China}
\email{xieg7@163.com}





\begin{abstract}
Let $G$ be a Hermitian type Lie group with the complexified Lie algebra $\mathfrak{g}$.  We use $L(\lambda)$ to denote a highest weight Harish-Chandra $G$-module with infinitesimal character $\lambda$. Let $w$ be an element in the Weyl group $W$. 
We use $L_w$ to denote a highest weight module with highest weight $-w\rho-\rho$. In this paper we  prove that there is only one Kazhdan--Lusztig right cell  such that the corresponding highest weight Harish-Chandra modules $L_w$ have the same associated variety. Then we
give a characterization for those $w$ such that  $L_w$ is a highest weight Harish-Chandra module and the associated variety of $L(\lambda)$ will be characterized by the information of the Kazhdan--Lusztig right cell containing some special $w_{\lambda}$. We also count the number of those highest weight Harish-Chandra modules $L_w$ in a given Harish-Chandra cell.


\end{abstract}

\maketitle
\setcounter{tocdepth}{1}
\tableofcontents
\section{Introduction}

Let $\mathfrak{g}$ be a finite-dimensional complex simple Lie algebra. In the 1960s,  Gelfand and  Kirillov \cite{Ge-Ki} introduced an quantity to measure the size of finitely generated $U(\mathfrak{g})$-modules, which is now known as the Gelfand--Kirillov dimension. For a finitely generated $U(\mathfrak{g})$-module $M$, Bernstein \cite{Be}  constructed a variety $V(M)$ in $\mathfrak{g}^{\ast}$, which is called the associated variety of $M$. Identifying $\mathfrak{g}^{\ast}$ with $\mathfrak{g}$ via the Killing form, the associated variety could also be viewed as a variety in $\mathfrak{g}$. It is proved that the dimension of $V(M)$ equals to the Gelfand--Kirillov dimension of $M$. The characterization of associated varieties has risen as a significant topic of infinite-dimensional representation theory of Lie algebras and Lie groups in recent years.

Fix a triangular decomposition $\mathfrak{g}=\mathfrak{n}\oplus \mathfrak{h} \oplus \bar{\mathfrak{n}}$ so that $\mathfrak{h}$ is a Cartan subalgebra of $\mathfrak{g}$ and $\mathfrak{b}=\mathfrak{h}+\mathfrak{n}$ is a Borel subalgebra.  Let $\Phi$ be the root system of $(\frg, \frh)$ with the positive system $\Phi^+$ and simple system $\Delta$. Denote by $W$ the Weyl group of $\Phi$. 
Let $G$ denote the algebraic adjoint group of $\mathfrak{g}$ and let $B$ denote the Borel subgroup of $G$ corresponding to $\mathfrak{b}$. For $\lambda\in\mathfrak{h}^{\ast}$, let $L(\lambda)$ be the highest weight module of $\mathfrak{g}$ with highest weight $\lambda-\rho$, where $\rho$ is  half the sum of positive roots. We use $L_w$ to denote the simple highest weight module with highest weight $-w\rho-\rho$.

 Joseph \cite{Jo84} proved that the associated variety $V(L_w)$ is a union of orbital varieties defined as follows. Let $\mathcal{O}\subseteq \mathfrak{g}$ be a nilpotent $G$-orbit. The irreducible components of $\overline{\mathcal{O}}\cap \mathfrak{n}$ are called {\it orbital varieties} of  $\mathcal{O}$. They all take the form $\mathcal{V}(w)=\overline{B(\mathfrak{n}\cap w\mathfrak{n})}$ for some $w\in W$. 
 The associated variety  $V(L_w)$ is  called {\it irreducible} if and only if it contains only one orbital variety. 
Tanisaki \cite{Ta} showed that there exist reducible associated varieties of the form $V(L_w)$ in type $B$ and $C$.
For a long time, people conjectured that all associated varieties of the form $V(L_w)$ are irreducible in the case of type $A$ (see \cite{BoB3} and \cite{Mel}). However,
Williamson \cite{Wi} showed that there exist counter examples. So the structure of $V(L_w)$ or $V(L(\lambda)) $ is still mysterious even in type $A$.

Suppose $G$ is a Hermitian type Lie group with Lie algebra $\mathfrak{g}_{\mathbb{R}}$ and complexified Lie algebra $\mathfrak{g}$. Let $K$ be the maximal compact Lie group of $G$ with Lie algebra $\mathfrak{k}_{\mathbb{R}}$ and complexified Lie algebra $\mathfrak{k}$. Let $\mathfrak{g} =\mathfrak{p}^-\oplus\mathfrak{k}\oplus\mathfrak{p}^+$ be the usual decomposition of $\mathfrak{g}$ as a $K_{\mathbb{C}}$-module.  When $L(\lambda)$ is a highest weight Harish-Chandra  module for $G$, the associated variety of $L(\lambda)$ is contained in $(\mathfrak{g}/(\mathfrak{k}+\mathfrak{p}^+))^*\simeq\mathfrak{p}^+$.  
It is well-known that the closures of the $K_{\mathbb{C}}$-orbits in $\mathfrak{p}^+$ form a chain
\begin{equation}\label{eqn:chain}
\{0\}=\overline{\mathcal{O}}_{0} \subsetneq  \overline{\mathcal{O}}_{1} \subsetneq  \cdots   \subsetneq  \overline{\mathcal{O}}_{r} =\mathfrak{p}^+\!,
\end{equation}
where $r$ is the real rank of $\mathfrak{g}_{\mathbb{R}}$, which is also equal to the rank of the symmetric space $G/K$.  It follows from \cite{Vo91} that 
\begin{equation}\label{eqn:av-intro}
V(L(\lambda))=\overline{\mathcal{O}}_{k}, \text{ for some }k=0,1,\dots,r.
\end{equation}

When $L(\lambda)$ is a unitary highest weight Harish-Chandra module, its associated variety is given By Bai-Hunziker \cite{BH} in a uniform formula.
Recently, a combinatorial algorithm was found to compute the Gelfand--Kirillov dimensions of all highest weight Harish-Chandra modules in \cite{BX,BXX}. From this algorithm, we can determine the associated varieties of all highest weight Harish-Chandra modules. A uniform formula for the associated varieties of highest weight Harish-Chandra modules is given in \cite{BHXZ}.

In this paper, we want to give some characterization for  highest weight Harish-Chandra modules by their associated varieties and Kazhdan--Lusztig right cells.

We refer to \cite{KL} or \S \ref{harishchandra} for the definition of Kazhdan--Lusztig right cell  and use $\stackrel{R}{\sim}$ to denote the right cell equivalence relation.

The  purpose of  this paper is to give a characterization of those $w$ such that $L_{w}$ is a Harish-Chandra module. Define
\begin{equation}\label{eqn:script-W}
  \mathcal{W}:=\{w\in W \mid -w\rho\text{ is $\Phi^+(\mathfrak{k})$-dominant}\}.
\end{equation}
It follows that  $\mathcal{W}=\{w\in W \mid \mathfrak{n} \cap \mathrm{Ad}(w) \mathfrak{n} \ \subseteq\ \mathfrak{p}^+\}$.   Note that the cardinality of $\mathcal{W}$ is $\#\mathcal{W}=\#(W/W(\mathfrak{k}))$, where $W(\mathfrak{k})$ is the Weyl group of $\mathfrak{k}$.

We will prove the following theorem.

\begin{Thm}[Theorem \ref{hermitian av} and Theorem \ref{thm:rightcell}]\label{HC-main}
	Suppose that $L_w$ and $L(\lam)$ are highest weight Harish-Chandra modules. Let  $y\in W$, then $w\stackrel{R}{\sim} y\stackrel{R}{\sim}w_{\lam}$ if and only if $V(L_w)=V(L_y)=V(L_{w_{\lam}})=V(L(\lam))=\overline{\mathcal{O}}_k$ for some $0\leq k\leq r$, where $w_{\lam}$ is the minimal length element of $W$ such that $w_{\lam}^{-1}\lam$ is antidominant.
\end{Thm}

From this theorem, we can see that there is only one right cell $\mathcal{C}_k$ (or Harish-Chandra cell) such that the corresponding highest weight Harish-Chandra modules have the same associated variety $\overline{\mathcal{O}}_k$. In order to find out this right cell, we will consider Wallach representations $L(-kc\xi+\rho)$ (see \cite{EHW} or Table \ref{constants} for the explicit value of the constant $c$ and $\xi$ is the fundamental weight orthogonal to $\Phi^+(\mathfrak{k})$ such that $(\xi,\beta^{\vee})=1$ with $\beta$ being the highest root of $\Phi^+$) for most types. Then we find out the minimal length element $w_{\lam}$ such that $w_{\lam}^{-1}\lam$ is antidominant for these Wallach representations. These special $w_{\lam}$'s are given in Theorem \ref{a-cell}, Theorem \ref{d-wlam}, Theorem \ref{wlam-2}, Theorem \ref{b-2-cell}, Theorem \ref{e6-lam}, Theorem \ref{e7-lam}, Theorem \ref{c-cell-even} and Theorem \ref{c-cell-odd}.  








In the case of $\mathfrak{sp}(n,\mathbb{R})$, $\mathfrak{so}(2,2n-1)$, $\mathfrak{e}_{6(-14)}$ and $\mathfrak{e}_{7(-25)}$, the elements of $\mathcal{W}$ or the right cell $\mathcal{C}_k$ are described in \cite{BZ,Zie18}.

We will count the size of each $\mathcal{C}_k$, especially for the case of $\mathfrak{su}(p,q)$,  $\mathfrak{so}^*(2n)$ and $\mathfrak{so}(2,2n-2)$. The results are given in Theorem \ref{size-b}, Theorem \ref{size-d} and Theorem \ref{size-d2} respectively. By using PyCox \cite{ge}, we can find out all the elements in a given right cell $\mathcal{C}_k$, see Remark \ref{pycox}.

This paper is organized as follows. In \S \ref{pre}, we give some necessary preliminaries about Kazhdan--Lusztig  cells and associated varieties of highest weight Harish-Chandra modules. In \S \ref{AV-cell}, we recall the results in \cite{BXX} about associated varieties  and prove Theorem \ref{HC-main}. In \S \ref{a-thm}-\ref{bc-thm}, we determine those special $w_\lambda$'s and count the size of each $\mathcal{C}_k$.

\section{Preliminaries}\label{pre}

In this section, we recall the definitions of Gelfand--Kirillov dimensions and associated varieties of highest weight modules.  More details can be found in Vogan \cite{Vo78}.

\subsection{GK dimension  and associated variety} Let $\mathfrak{g}$ be a finite-dimensional simple Lie algebra. Let $M$ be a $U(\mathfrak{g})$-module generated by a finite-dimensional subspace $M_0$. Let $\{U_{n}(\mathfrak{g})\}_{n\geq 0}$ be the standard filtration of $U(\mathfrak{g})$. The graded module of $M$ is defined by $\text{gr} (M)=\bigoplus\limits_{n=0}^{\infty} \text{gr}_n M$ (here $M_n=U_n(\mathfrak{g})M_0$, $\text{gr}_n M=M_n/{M_{n-1}}$), which is a finitely generated $\text{gr}(U(\mathfrak{g}))$($\approx S(\mathfrak{g})$)-module.
\begin{definition}   Then the \textit{Gelfand--Kirillov dimension} of $M$  is defined by$$
	\operatorname{GKdim} M = \overline{\lim\limits_{n\rightarrow \infty}}\frac{\log\dim( U_n(\mathfrak{g})M_{0} )}{\log n}.
	$$
\end{definition}
It is well-known that the above definition  is independent of  the choice of  $M_0$ (see \cite{NOT}).

The \textit{associated variety} of $M$ is defined by
$$
V(M):=\{X\in \mathfrak{g}^* \mid p(X)=0\ \text{for~ all~} p\in \operatorname{Ann}_{S(\mathfrak{g})}(\operatorname{gr} M)\}.
$$
From  \cite{NOT}, we have $\dim V(M)=\operatorname{GKdim} M$.

Let $G$ be a simple real Lie group with the complexified Lie algebra $\mathfrak{g}$ and a maximal compact subgroup $K$ with complexified Lie algebra $\mathfrak{k}$, so that $G/K$ is a Hermitian symmetric space. Let $\mathfrak{g} =\mathfrak{p}^-\oplus\mathfrak{k}\oplus\mathfrak{p}^+$ be the usual decomposition of $\mathfrak{k}$-modules, where $\mathfrak{p}^\pm$ are abelian. Then $\mathfrak{q}=\mathfrak{k}\oplus\mathfrak{p}^+$ is a maximal parabolic subalgebra of $\mathfrak{g}$ with nilradical $\mathfrak{p}^+$. Let $\mathfrak{h}$ be a Cartan subalgebra of $\mathfrak{k}$ and $\mathfrak{g}$. Let $\Phi$ be the root system of $(\frg, \frh)$ with the positive system $\Phi^+$ and simple system $\Delta$. 

Let $F(\lambda)$ be the finite-dimensional irreducible $\mathfrak{k}$-module of highest weight $\lambda-\rho \in \mathfrak{h}^*$. By letting $\mathfrak{p}^+$ act trivially on $F(\lambda)$, we may consider $F(\lambda)$ as a module of $\mathfrak{q}$. The {\it generalized Verma module} $N(\lambda)$ is defined to be
\[
N(\lambda):=U(\mathfrak{g})\otimes_{U(\mathfrak{q})}F(\lambda).
\]
The simple quotient of $N(\lambda)$ is denoted by $L(\lambda)$. We call it  a \textit{highest weight Harish-Chandra module}.   The classification of (unitary) highest weight Harish-Chandra modules had been done by Enright--Howe--Wallach \cite{EHW} and Jakobsen \cite{Ja1} independently.


\begin{Pro}[Vogan \cite{Vo91}]\label{p3.1}
	Let $L(\lambda)$ be a highest weight Harish-Chandra module. Then we have $V(L(\lambda))=\overline{\mathcal{O}}_{k(\lambda)}$ for some integer $k(\lambda)\geq 0$.
\end{Pro}


Let $\beta$ denote the highest  root of the root system $\Phi$ and $\rho$ be half the  sum of positive roots in $\Phi^+$.  Let $(-, -): \mathfrak{h} \times \mathfrak{h}^* \to \mathbb{C}$ be the canonical pairing.
The dual Coxeter number $h^\lor$ of the root system $\Phi$  is defined by
$$h^\lor:=(\rho,\beta^\lor)+1.$$
 The Wallach set is a special set consisting of unitary highest weight modules with highest weights:
 $$W_s=\{-kc\xi\mid k\in \mathbb{Z}  \text{~and~} 0\leq k\leq r-1 \},$$
 where $\xi$ is the fundamental weight orthogonal to $\Phi(\mathfrak{k})$ and $(\xi,\beta^{\vee})=1$, and $c$ is a constant associated to the Hermitian symmetric space $G/K$.  From  \cite{BH},
 we have the following table.




\begin{table}[htbp]
	\centering
	\renewcommand{\arraystretch}{1.4}
	\setlength\tabcolsep{5pt}
	
{\caption{Some constants associated to $G/K$}
\label{constants}
\begin{tabular}{cccc}
\hline  $\mathfrak{g}_{\mathbb{R}}=\operatorname{Lie}(G)$ &   $r$ & $c$ & $h^\vee-1$ \\ \hline 
 $\mathfrak{su}(p,q)$, $n=p+q-1$ & $\min\{p,q\}$ &  $1$ & $n$  \\  
    $\mathfrak{sp}(n,\mathbb{R})$  & $n$ &   $1/2$ & $n$   \\  
    $\mathfrak{so}^{*}(2n)$, $n=2m$  & $m$  & $2$ & $2n-3$ \\ 
   $\mathfrak{so}^{*}(2n)$, $n=2m+1$  & $m$  & $2$ & $2n-3$ \\ 
   $\mathfrak{so}(2,2n-1)$  & $2$  &  $n-3/2$ & $2n-2$ \\ 
  $\mathfrak{so}(2,2n-2)$  & $2$ &  $n-2$& $2n-3$ \\ 
  $\mathfrak{e}_{6(-14)}$   & $2$ &  $3$ & $11$ \\  
  $\mathfrak{e}_{7(-25)}$  & $3$ &  $4$ & $17$ \\
\hline

\end{tabular}}

\end{table}

When $L(\lambda)$ is a unitary highest weight Harish-Chandra module, the explicit value of $k(\lambda)$ had been given by Bai-Hunziker \cite{BH} in a simple uniform way. 

	\begin{Pro}[{\cite{BH}}]\label{C: dimYk}
		Suppose $L(\lambda)$ is a unitary highest weight Harish-Chandra module with highest weight $\lambda-\rho$. We denote $z=z(\lambda)=(\lambda,\beta^{\vee})$, then
		$$\operatorname{GKdim} L(\lambda)=
		\left\{
		\begin{array}{ll}
		rz_{r-1}, & \text{if~}  z<z_{r-1}\\
		kz_{k-1}, & \text{if~}  z=z_{k}=(\rho,\beta^{\vee})-kc, 1\leq k\leq r-1\\
		0, & \text{if~} z=z_{0}=(\rho,\beta^{\vee}).
		\end{array}
		\right.$$
		
		Denote $k=k(\lambda):=-\frac{(\lambda-\rho,  \beta^{\vee} )}{c}$.   Then
		\begin{enumerate}
			\item  If $k>r-1$, we have $\operatorname{GKdim} L(\lambda)=rz_{r-1}=\frac{1}{2}\dim(G/K).$
			\item If $0\leq k\leq r-1$, then $k$ is a non-negative integer and $$\operatorname{GKdim} L(\lambda)=k((\rho, \beta^{\vee})-(k-1)c)=kz_{k-1}=\dim \overline{\mathcal{O}}_{k(\lambda)}.$$
		\end{enumerate}
		The associated variety of $L(\lambda)$ is $\overline{\mathcal{O}}_{k(\lambda)}$.

	\end{Pro}

When $L(\lambda)$ is just a highest weight Harish-Chandra module (unitary or not), the explicit value of $k(\lambda)$ had been given in \cite{BX, BXX}.

\subsection{Kazhdan--Lusztig cell and Harish-Chandra cell}\label{harishchandra}

The Weyl group $ W  $ of $ \mf{g} $ is a Coxeter group generated by $ S=\{s_\al\mid\al\in\Delta \} $. Let $\ell(-)$ be the \textit{length function} on $W$. The Hecke algebra $ \mc{H} $ over $ \mathcal{A} :=\mathbb{Z}[v,v^{-1}]$ is generated by $ T_w $, $ w\in W $ with relations \[
T_wT_{w'}=T_{ww'} \text{ if }\ell(ww')=\ell(w)+\ell(w'),
\]
\[
\text{and }(T_s+v^{-1})(T_s-v)=0 \text{ for any }s\in S.
\]
The unique elements $ C_w $ such that
\[
\overline{C_w}=C_w, \text{~and~}C_w\equiv T_w  ~(\mathrm{mod}{~\mc{H}_{<0}})
\]
is known as the Kazhdan--Lusztig basis of $ \mc{H} $, where $ \bar{\,} :\mc{H}\rar\mc{H}$ is the bar involution such that $$ \bar{q}=q^{-1},~ \overline{T_w} =T_{w^{-1}}^{-1}, \text{~and~}  \mc{H}_{<0}=\bigoplus_{w\in W}\mathcal{A}_{<0}T_w $$ with $ \mathcal{A}_{<0}=v^{-1}\mathbb{Z}[v^{-1}]$.

If $ C_y $ occurs in the expansion of $ hC_w $ (resp. $C_wh$) with respect to the KL-basis for some $ h\in\mc{H} $,  we write $ y\leftarrow_L w $ (resp. $ y\leftarrow_R w $). Extend $ \leftarrow_L $ (resp. $ \leftarrow_R $) to a preorder $ \leq_L $ (resp. $\leq _R$) on $ W $. For $x, w\in W$, write $x \leq_{LR} w$ if there exists $x=w_0, w_1, \cdots, w_n=w$ such that for every $0\leq i<n$ we have  $w_i\leq_L w_{i+1}$ or $w_i\leq_R w_{i+1}$. Let $\stackrel{L}{\sim}$, $\stackrel{R}{\sim}$, $\stackrel{LR}{\sim}$ be the equivalence relations associated with $\leq_L$, $\leq_R$, $\leq_{LR}$. The equivalence classes on $W$ for $\stackrel{L}{\sim}$, $\stackrel{R}{\sim}$, $\stackrel{LR}{\sim}$ are called \textit{left cells}, \textit{right cells} and \textit{two-sided cells} respectively. From \cite[\S 8.1]{Lus03}, we have $x\stackrel{R}{\sim}y$ if and only if $x^{-1}\stackrel{L}{\sim}y^{-1}$.

Now we recall the concept of
Harish-Chandra cells, which is defined by Barbasch--Vogan \cite{BV83-book}.   Let $X,Y$ be two irreducible (Harish-Chandra) modules  with the same infinitesimal character. We write $X>Y$ if there exists a finite-dimensional $\mathfrak{g}$-module $F$ such that $Y$ appears a sub-quotient of $X\otimes F$. We write $X\sim Y$ if both $X>Y$ and $Y>X$. The equivalence classes for the relation $\sim$ are called cells of (Harish-Chandra) modules or (Harish-Chandra) cells.
\begin{Pro}[{\cite[Prop. 2.9]{BV83}}]\label{HC-cell}
    Two highest weight Harish-Chandra modules $L_w$ and $L_y$ belong to the same Harish-Chandra cell with infinitesimal character $\rho$ if and only if  $w\stackrel{R}{\sim} y$. 
\end{Pro}

 It follows
 from the definition that if a Harish-Chandra  cell  contains one highest weight Harish-Chandra module, then all representations in this cell are highest weight Harish-Chandra modules. Also the associated variety is constant on any given Harish-Chandra cell.
 
  From \cite[\S 3.2]{BZ}, we know that a Harish-Chandra cell $\mathcal{C}_{HC}$ will determine a representation of $W$  which is spanned
by $\mathcal{C}_{HC}$ and will be denoted by $V(\mathcal{C}_{HC})$. Suppose $\mathcal{O}^{\mathbb{C}}$ is the complex nilpotent orbit appeared in the associated variety of the annihilator of the highest weight Harish-Chandra module $L(\lambda)$. Then  the  irreducible $W$-module $\pi(\caO^\C,1)$  for 
  $\mathcal{O}^{\mathbb{C}}$ via the Springer correspondence is a $W$-submodule of $V(\mathcal{C}_{HC})$. Thus we have $\#\mathcal{C}_{HC}=\dim V(\mathcal{C}_{HC})\geq \dim \pi(\caO^\C,1)$.

\section{Associated varieties of highest weight modules}\label{AV-cell}

%
%
In this section, we  recall the relation between associated varieties and cells.

\subsection{Right cells and associated varieties}

 Recall that any subset $I\subset\Delta$ generates a subsystem $\Phi_I\subset\Phi$ with corresponding parabolic subalgebra $\mathfrak{p}_I=\mathfrak{l}_I\oplus \mathfrak{u}_I\supset\frb$ of $\frg$.

We say $\lambda\in\frh^*$ is \emph{$\Phi_I^+$-dominant} if and only if $(\lambda, \alpha^\vee)\in\mathbb{Z}_{>0}$ for all $\alpha\in I$. So $L(\lambda)$ is a highest weight Harish-Chandra module if and only if $\lambda$ is $\Phi^+(\mathfrak{k})$-dominant.

\begin{Pro}[{\cite{BMXX}}]\label{HC}
	Let $\frg$ be a simple complex Lie algebra and $\lambda\in \frh^*$. Then $\lambda$ is $\Phi_I^+$-dominant if and only if $V(L(\lambda))\subset\fru_I$.
\end{Pro}

Denote $I_w=\Ann(L_w)$. By Borho-Brylinski \cite{BoB1} and Joseph \cite{Jo85}, we know that $V(I_w)=V(U(\mathfrak{g})/I_w)$ is the closure of a single special (in the sense of Lusztig \cite{Lu79}) nilpotent orbit. From \cite{KL} and \cite{BV83}, there is a bijection between special nilpotent orbits of $\frg$ and two-sided cells of the Weyl group $W$, see also \cite{Ta}. The special orbits are listed in Collingwood-McGovern \cite{CM}. For a nilpotent orbit $\mathcal{O}$, denote by $\mathrm{Irr}(\overline{\mathcal{O}}\cap \mathfrak{n})$ the set of irreducible components in $\overline{\mathcal{O}}\cap \mathfrak{n}$.  Each irreducible component of $\overline{\mathcal{O}}\cap \mathfrak{n}$ is called an {\it orbital variety}, which can be written as $\mathcal{V}(w):=\overline{B(\mathfrak{n}\cap w\mathfrak{n})}$ for some $w\in W$ \cite{Ta}.  

\begin{Pro}[{\cite[Lem. 6.6]{Jo84}} and  {\cite[Cor. 6.3]{BoB3}}]\label{Avconstant}
$V(L_w)$ is constant on each KL right cell.	
	
\end{Pro}

The following result was conjectured by Tanisaki \cite[Conj.3.4]{Ta}. The case of type $A$ was proved by Borho-Brylinski \cite{BoB3}, while the case of the other classical types was proved by McGovern \cite{Mc00}.

\begin{Pro}\label{Ta}
	Suppose that $\frg$ is classical. Let  $\mathscr{C}$ be the two-sided cell corresponding to a special nilpotent orbit $\mathcal{O}$. Denote by $\mathscr{C}_R$  the set of right cells contained in $\mathscr{C}$. Then there exists a bijection from $\mathscr{C}_R$ to $\mathrm{Irr}(\overline{\mathcal{O}}\cap \mathfrak{n})$ $(w\rightarrow Y_w)$ and an ordering $\prec$ on $\mathrm{Irr}(\overline{\mathcal{O}}\cap \mathfrak{n})$ such that $V(L_w)=Y_w \cup \tilde Y_w$, where $\tilde{Y}_w$ is a union of some orbital varieties $Y$ in $\mathrm{Irr}(\overline{\mathcal{O}}\cap \mathfrak{n})$  with $Y\prec Y_w$.
\end{Pro}

\begin{Rem}\label{conjec} $Y_w$ depends on the right cell containing $w$. It may not be $\mathcal{V}(w)$.
	In Tanisaki's original conjecture \cite[Conj.3.4]{Ta}, $\frg$ is not necessarily classical, where the cases of type $E_6$ and $G_2$ were proved by himself. For the case of type $F_4$, Tanisaki showed that the conjecture is true for nine special nilpotent orbits among all the eleven ones. The cases of type $E_7, E_8, F_4$ are still unknown.
\end{Rem}

\begin{Thm}\label{right cell-equi}
	Suppose $\mathfrak{g}$ is of classical type. Let  $w,y\in W$, then $w\stackrel{R}{\sim} y$ if and only if $V(L_w)=V(L_y)$.
\end{Thm}

\begin{proof}
	If $w\stackrel{R}{\sim} y$, we can get $V(L_w)=V(L_y)$ by Proposition \ref{Avconstant}. Now assume that $V(L_w)=V(L_y)$. In view of \cite[Cor. 4.11]{BoB1} and \cite[Thm. 10.2.2]{CM}, $V(I_w)=V(I_y)$ is the closure of a nilpotent orbit $\mathcal{O}$. Let $\mathscr{C} $ be the two-sided cell corresponding to $\mathcal{O}$ in the Springer correspondence. With $V(L_w)=V(L_y)$, Proposition \ref{Ta} yields $Y_w \cup \tilde{Y}_w=Y_y \cup \tilde{Y}_y$. This forces $Y_w=Y_y$. Hence $w\stackrel{R}{\sim} y$. 
	
\end{proof}

\begin{Rem}\label{EG} Theorem \ref{right cell-equi} will hold for all  simple Lie algebras once Tanisaki's original conjecture \cite[Conj. 3.4]{Ta} is proved completely.
\end{Rem}


\begin{Pro}[{\cite[Prop. 4.2]{BHXZ}}]\label{simplecell}
    Suppose $\Phi$ is simply laced and let $w \in  \mathcal{W}$. For $k=0,1,\dots,r$, the set $$\{L_w\mid  w\in\mathcal{W} \text{ and }V(L_w)=\overline{\mathcal{O}}_k\}$$ is a Harish-Chandra cell and the corresponding $W$-module is irreducible. In fact, we have $V(L_w)=\mathcal{V}(w)$ for any $w\in \mathcal{W}$.
\end{Pro}

\begin{Thm}\label{hermitian av}
	Suppose that $L_w$ is a highest weight Harish-Chandra module. Let  $y\in W$, then $w\stackrel{R}{\sim} y$ if and only if $V(L_w)=V(L_y)$.
\end{Thm}

\begin{proof}
From Theorem \ref{right cell-equi} and Remark \ref{conjec}, we only need to prove the case of $\mathfrak{g}=E_7$. Now we assume that $V(L_y)=V(L_w)\subset \mathfrak{p}^+$. Then by Proposition \ref{HC}, we have  $y,w\in \mathcal{W}$ and $L_y$ is also a highest weight Harish-Chandra module.



Now $r=3$,
thus  $V(L_y)=V(L_w)=\overline{\caO}_i$ for some $0\leq i\leq 3$. By Proposition \ref{HC-cell} and \ref{simplecell}, we must have  $w\stackrel{R}{\sim} y$ since there is only one right cell (Harish-Chandra cell) for each case.

\end{proof}
\begin{example}\label{e7-cell-eq}
    
For  $\mathfrak{g}_{\mathbb{R}}=\fre_{7(-25)}$,  the cardinality of $\caW$ is 56 and $r=3$.  So there are four $K_{\mathbb{C}}$-orbits $\caO_0,\caO_1$,$\caO_2$ and $\caO_3$ in $\frp^+$. From \cite[\S 3]{BZ}, we know that there is an irreducible $W$-module $\pi(\caO_i^\C,1)$  for 
 each complex nilpotent $G$-orbit $\mathcal{O}_i^{\mathbb{C}}$ via the Springer correspondence.
Table \ref{e7} contains some data for $\fre_{7(-25)}$ (using the notation of \cite{Ca85} where $\phi_{a,b}$ denotes the corresponding  irreducible $W$-module $\pi(\caO_i^\C,1)$ which has dimension $a$ and degree $b$).

\begin{table}[htp]
\centering
\renewcommand{\arraystretch}{1.5}
\setlength\tabcolsep{10pt}
\caption{$K_{\mathbb{C}}$-orbits for $\mathfrak{g}_{\mathbb{R}}=\fre_{7(-25)}$}\label{e7} 
\begin{tabular}{cccc}
\hline
$i$  &  $\dim(\caO_i)$    & $\pi(\caO_i^\C,1)$  &  $\caO_i^\C$  \\
\hline 
$0$  & $0$    & $\phi_{1,63}$  & $0$   \\
$1$  & $11$ &   $\phi_{7,46}$  &  $A_1$  \\
$2$  & $16$ &   $\phi_{27,37}$  &  $2A_1$\\
$3$  & $16$ &   $\phi_{21,36}$  &  $(3A_1)''$\\
\hline
\end{tabular}

\end{table}
There are four Harish-Chandra cells $\mathcal{C}_0=\{\C \}, \mathcal{C}_1$, $\mathcal{C}_2$ and $\mathcal{C}_3$.  We have $\#\mathcal{C}_0=1$. By \S \ref{harishchandra} and Table \ref{e7}, we have $\#\mathcal{C}_1\geq 7$, $\#\mathcal{C}_2\geq 27$ and $\#\mathcal{C}_3\geq 21$, but $\# \caW=56$, so equality holds.  Thus Proposition  \ref{simplecell} holds for this case.
\end{example}

\subsection{Translation functor and associated varieties}
 For $\mu\in\mathfrak{h}^*$, define 
\begin{equation*}
\Phi_{[\mu]}:=\{\alpha\in\Phi\mid (\mu, \alpha^\vee)\in\mathbb{Z}\}.
\end{equation*}
 Set 
\[
W_{[\mu]}:=\{w\in W\mid w\mu-\mu\in \mathbb{Z}\Phi\}.
\]
Then $\Phi_{[\mu]}$ is a root system with Weyl group $W_{[\mu]}$. 
 Let $\Delta_{[\mu]}$ be the simple system of $\Phi_{[\mu]}$. Set $J=\{\alpha\in\Delta_{[\mu]}\mid(\mu, \alpha^\vee)=0\}$. Denote by $W_J$ the Weyl group generated by reflections $s_\alpha$ with $\alpha\in J$. Let $\ell_{[\mu]}$ be the length function on $W_{[\mu]}$. Thus $\ell_{[\mu]}=\ell$ when $\mu$ is integral. Put
\begin{equation*}\label{ceq1}
	W_{[\mu]}^J:=\{w\in W_{[\mu]}\mid \ell_{[\mu]}(ws_\alpha)=\ell_{[\mu]}(w)+1\ \mbox{for all}\ \alpha\in J\}.
\end{equation*}
Thus $W_{[\mu]}^J$ consists of the shortest representatives of the cosets $wW_J$ with $ w\in W_{[\mu]} $. When $\mu$ is integral, we simply write $W^J:=W_{[\mu]}^J$ .

A weight $ \mu\in\hs $ is called \textit{anti-dominant} if $ ({\mu},{\alpha^{\vee}}) \notin\mathbb{Z}_{>0}$ for all $ \al\in\Phi^+ $. For any $\lambda\in\mathfrak{h}^*$, there exists a unique anti-dominant weight $\mu\in\hs$ and a unique $w_{\lambda}\in W_{[\mu]}^J$ such that $\lambda=w_{\lambda}\mu$. 

\begin{Pro}[{\cite[Prop. 3.5]{Hum08}}]\label{anti}
    Let $\lambda\in \mathfrak{h}^*$, with corresponding root system $\Phi_{[\lambda]}$ and Weyl group $W_{[\lambda]}$. Let $\Delta_{[\lambda]}$ be the simple system of $\Phi_{[\lambda]} \cap \Phi^+$  in $\Phi_{[\lambda]}$.  Then $\lambda$ is antidominant if and only if one of the
following three equivalent conditions holds:
\begin{enumerate}
    \item $(\lambda, \alpha^\vee)\leq 0$ for all $\alpha \in \Delta_{[\lambda]}$;
    \item $\lambda\leq s_{\alpha}\lambda$ for all $\alpha \in \Delta_{[\lambda]}$;
    \item  $\lambda\leq w\lambda$ for all $w \in W_{[\lambda]}$.
\end{enumerate}
Therefore, there is a unique antidominant weight in the orbit $W_{[\lambda]}\lambda$.
\end{Pro}

Let  $ \mathcal{O} $ be the BGG category consists of $ \mathfrak{g} $-modules which are  semisimple as $ \mathfrak{h} $-modules, finitely generated as $ U(\mathfrak{g}) $-modules and locally $ \mathfrak{n} $-finite. A \textit{block} of $ \mathcal{O} $ is an indecomposable summand of $ \mathcal{O} $ as an abelian subcategory. Let $ \mathcal{O} _\lambda $  be the block containing  the simple module $ L(\lambda) $.

For two integral weights $ \lambda,\mu \in \mathfrak{h}^*$, we set $ \gamma=\mu-\lambda $. Then we can find a weight $ \bar{\gamma} \in W\gamma$ such that $ \langle {\bar{\gamma}},{\alpha} \rangle \geq 0$ for all $ \alpha \in\Delta^+$. Then $ L(\bar{\gamma}) $ is finite dimensional. The Jantzen's \textit{translation functor} (see Jantzen \cite{Jantzen79Lec750} or Humphreys \cite{Hum08}): $$  T_\lambda^\mu : \mathcal{O}_\lambda \rightarrow \mathcal{O}_\mu $$ is   an exact functor given by $ T_\lambda^\mu(M):=\text{ pr}_\mu (L(\bar{\gamma})\otimes M) $, where $ M\in   \mathcal{O}_\lambda $ and $ \text{ pr}_\mu $ is the natural projection $ \mathcal{O} \rightarrow \mathcal{O}_\mu $.

From Borho-Brylinski-MacPherson \cite[Lem. 5.2]{BBM}, we know that the associated variety of an irreducible $\mathfrak{g}$-module is invariant under Jantzen's translation functor. In particular, we have the following lemma from Bai-Xie \cite[Cor. 3.3]{BX} and  \cite[Cor. 4.3.1]{BoB3}.

\begin{Lem}\label{l3.3}
	For any integral weight $\lambda$, we have
	\[
	T_{-w_{\lambda}\rho-\rho}^\lambda (L_{w_{\lambda}})=L(\lambda)
	\]
	and $V(L(\lambda))=V(L_{w_{\lambda}})\supseteq\caV(w_\lambda)$, where $w_{\lambda}\in W $ is the unique element of minimal length such that $ w_{\lambda}^{-1}\lambda $ is antidominant.
\end{Lem}




\begin{Thm}\label{thm:rightcell}
	Suppose that $L(\lambda)$ is a highest weight Harish-Chandra module with integral highest weight $\lambda-\rho$. Suppose $\mu\in\frh^*$ is also integral. Then $V(L(\lambda))=V(L(\mu))$ if and only if $w_\lambda\stackrel{R}{\sim}w_\mu$.
\end{Thm}
\begin{proof}
    From $V(L_{w_{\lambda}})=V(L(\lambda))=V(L(\mu))=V(L_{w_{\mu}})\subset \mathfrak{p}^+$,  we will have $V(L_{w_{\lambda}})=V(L_{w_{\mu}})\subset \mathfrak{p}^+$.
    Then by Theorem \ref{hermitian av}, we have $w_\lambda\stackrel{R}{\sim}w_\mu$.

    The other direction follows from Proposition \ref{Avconstant}.
    
\end{proof}

Recall that the Wallach representations $L(-kc\zeta+\rho)$ are unitarizable highest weight Harish-Chandra modules with $V(L(-kc\zeta+\rho))=\overline{\mathcal{O}}_k$ for $k=0,1,\ldots, r$ (see Joseph \cite{Jo92} or Bai-Hunziker \cite{BH}). Thus we only need to consider these Wallach representations.

\subsection{Associated varieties of highest weight HC modules}

For a Young diagram $P$ with shape $p$, we use $ (k,l) $ to denote the box in the $ k $-th row and the $ l $-th column.
We say the box $ (k,l) $ is \emph{even} (resp. \emph{odd}) if $ k+l $ is even (resp. odd). Let $ p_i ^{\ev}$ (resp. $ p_i^{\od} $) be the number of even (resp. odd) boxes in the $ i $-th row of the Young diagram $ P $.
One can easily check that \begin{equation}\label{eq:ev-od}
	p_i^{\ev}=\begin{cases}
		\left\lceil \frac{p_i}{2} \right\rceil,&\text{ if } i \text{ is odd},\\
		\left\lfloor \frac{p_i}{2} \right\rfloor,&\text{ if } i \text{ is even},
	\end{cases}
	\quad p_i^{\od}=\begin{cases}
		\left\lfloor \frac{p_i}{2} \right\rfloor,&\text{ if } i \text{ is odd},\\
		\left\lceil \frac{p_i}{2} \right\rceil,&\text{ if } i \text{ is even}.
	\end{cases}
\end{equation}
Here for $ a\in \mathbb{R} $, $ \lfloor a \rfloor $ is the largest integer $ n $ such that $ n\leq a $, and $ \lceil a \rceil$ is the smallest integer $n$ such that $ n\geq a $. For convenience, we set
\begin{equation*}
	p^{\ev}=(p_1^{\ev},p_2^{\ev},\cdots)\quad\mbox{and}\quad p^{\od}=(p_1^{\od},p_2^{\od},\cdots).
\end{equation*}

For  $ x=(x_1,x_2,\cdots,x_n)\in \mathrm{Seq}_n (\Gamma) $, we denote 
\begin{align*}
		{x}^-=&(x_1,x_2,\cdots,x_{n-1}, x_n,-x_n,-x_{n-1},\cdots,-x_2,-x_1),\\
		{}^-{x}=&(-x_n,-x_{n-1},\cdots, -x_2,-x_1,x_1,x_2,\cdots, x_{n-1}, x_n).
	\end{align*}

For $i<j\in \mathbb{Z}$, denote $[i,j]=\{i,i+1,\cdots,j-1,j\}$.
When the root system $\Phi=B_n$ or $C_n$ with $n>1$, the Weyl group is $W=W_n$, where $ W_n $ is the group consisting of permutations $ w $ of the set $[-n, n]$
	such that $ w(-i)=-w(i) $ for all $ i \in[1,n]$. Denote $s_i:=s_{\alpha_i}$ for $1\leq i\leq n-1$. Let $ t\in W_n$ be the element with $ t=(-1,2, \dots, n):=(1,-1)$. Then $ (W_n , T_n)$ is a Coxeter system, with $T_n=\{t,s_1,\dots s_{n-1}\} $.  Let $ u\in W_n$ be the element with $ u=(-2,-1,3, \dots, n):=(1,-2)$. Then $ (W'_n , T'_n)$ is a Coxeter system, with $T'_n=\{u,s_1,\dots s_{n-1}\} $.
	
The Weyl group $ (W,S) $ of $ B_n $ is isomorphic to $ (W_n, T_n) $ via
\[
s_{\ep_i-\ep_{i+1}}\mapsto s_{n-i},0\leq  i\leq n-1,\text{ and } s_{\ep_n}\mapsto t.
\]

The Weyl group $ (W,S) $ of $ C_n $ is isomorphic to $ (W_n, T_n) $ via
\[
s_{\ep_i-\ep_{i+1}}\mapsto s_{n-i},0\leq  i\leq n-1,\text{ and } s_{2\ep_n}\mapsto t.
\]

The Weyl group $ (W,S) $ of  $ D_n $ is isomorphic to $ (W_n', T_n') $  via
\begin{equation}\label{isom}
    s_{\ep_i-\ep_{i+1}}\mapsto s_{n-i},1\leq  i\leq n-1,\text{ and } s_{\ep_{n-1}+\ep_n}\mapsto u.
\end{equation}

\begin{Lem}\label{lambda-w}
Assume that $ \lambda $ is an integral weight with $  \lambda=(\lam_1,\cdots,\lam_n)$.
\begin{itemize}
\item [(1)]There is a unique element $ w\in W_n $ such that
\begin{itemize}
\item[(i)] if $ \lam_i\neq 0 $ then $ \sigma_i' $ has the same sign with $ \lam_i $;
\item [(ii)]if $\lam_i= 0 $ then $ \sigma_i' <0$;
\item[(iii)] if $ \lam_i<\lam_j $, then $ \sigma_i'<\sigma_j'  $;
\item[(iv)] if $\lam_i=\lam_j $ and $ i<j $, then $ \sigma_i'<\sigma_j'  $.
\end{itemize}
where $ (\sigma_1',\sigma_2',\cdots, \sigma_n')=(-w^{-1}(n), -w^{-1}(n-1),\cdots, -w^{-1}(1)) $.
\item [(2)] The element $ w $ in (1) is precisely the minimal length element such that $ w^{-1}\lambda $ is antidominant. We denote is by $w_{\lambda}$.

   \item [(3)] Let $ w $ be the element in (1) and $t=(1,-1)\in W_n$. Then one of $ w $ and $ wt $ belongs to $ W_n' $, and we denote it by $ w_{\lambda} $. Then $  w_{\lambda}  $ is precisely the minimal length element such that $  w_{\lambda}^{-1}\lambda $ is antidominant.

\end{itemize}
\end{Lem}
\begin{proof}
See the arguments in the proof of \cite[Lem. 5.2 and Lem. 5.3]{BXX}

\end{proof}

\begin{Rem}

Let $ \mu $ be the antidominant weight in the orbit $ W\lambda $, with $ \mu=(\mu_1,\cdots,\mu_2) $.
One can obtain a (reduced) string diagram by
connecting the equal numbers, in different rows and  one time for each number, of the diagram\[
\begin{matrix}
\lam_1&\lam_2&\cdots &\lam_{n-1}&\lam_n &-\lam_{n}&-\lam_{n-1}&\cdots&-\lam_2&-\lam_1\\\mu_1&\mu_2&\cdots &\mu_{n-1}&\mu_n&-\mu_{n} &-\mu_{n-1}&\cdots&-\mu_2&-\mu_1.
\end{matrix}
\]
Then such a  string diagram represents an element $ w\in W_n $. By construction, we have $ w^{-1}\lambda=\mu $, and $ w $ satisfies (i) and (iii) of (1). If in the same time we require $ w $  to satisfy (ii) and (iv),  then the string diagram has a minimal number of crossings, and hence $ w $ will be the  element with  minimal length such that $ w^{-1}\lambda =\mu$. 

Let $w$ be the element found in (1) of Lemma \ref{lambda-w}. We may have $w\notin W'_n$ but $wt\in W'_n$. By \cite[Rem. 4.11]{BXX}, we have $(wt)(k)=-w(k)$ when $|w(k)|=1$ and $(wt)(k)=w(k)$ when $|w(k)|>1$. For example, if $w=(2,1,-3,4)$, we will have $wt=(2,-1,-3,4)\in W'_n$.
\end{Rem}

\begin{example}

Let $ \mf{g}=\mathfrak{so}(11, \mathbb{C})$ and $ \lambda=(2,-1,1,-3,1) $. Then from Proposition \ref{anti}, the antidominant weight in $ W\lambda $ is $ \mu= (-3,-2,-1,-1,-1) $. We have a string diagram:
\begin{center}
\begin{tikzpicture}
 \draw(0.2,-0.2) node{$ \mu^-: $};
\filldraw (1,0) \nodd node[below]{$ -3 $};  \draw (1,0)--(4,3);
\filldraw (2,0) \nodd node[below]{$ -2$};  \draw (2,0)--(10,3);
\filldraw (3,0) \nodd node[below]{$ -1 $};	\draw (3,0)--(2,3);
\filldraw (4,0) \nodd node[below]{$ -1 $};	\draw (4,0)--(6,3);
\filldraw (5,0) \nodd node[below]{$ -1 $};	\draw (5,0)--(8,3);

\filldraw (6,0) \nodd node[below]{$ 1$};	\draw (6,0)--(3,3);
\filldraw (7,0) \nodd node[below]{$ 1 $};	\draw (7,0)--(5,3);
\filldraw (8,0) \nodd node[below]{$ 1 $};	\draw (8,0)--(9,3);
\filldraw (9,0) \nodd node[below]{$ 2 $};	\draw (9,0)->(1,3);
\filldraw (10,0) \nodd node[below]{$ 3 $}; \draw (10,0)->(7,3);
\draw[dashed](5.5,0)--(5.5,3);
 \draw (0.2,3.3) node{$ \lambda^-: $};
\filldraw (1,3) \nodd node[above]{$2 $};
\filldraw (2,3) \nodd node[above]{$ -1$};
\filldraw (3,3) \nodd node[above]{$ 1 $};
\filldraw (4,3) \nodd node[above]{$ -3 $};
\filldraw (5,3) \nodd node[above]{$ 1 $};

\filldraw (6,3) \nodd node[above]{$ -1$};
\filldraw (7,3) \nodd node[above]{$ 3 $};
\filldraw (8,3) \nodd node[above]{$ -1 $};
\filldraw (9,3) \nodd node[above]{$ 1 $};
\filldraw (10,3) \nodd node[above]{$ -2 $};
\end{tikzpicture}.
\end{center}
By Lemma \ref{lambda-w}, this diagram represents an element $ w=(-3,-1,4,-5,2)\in W_5 $.
    
\end{example}

\begin{Lem}[{\cite[Lem. 5.2]{BXX}}]
Suppose $\Phi=B_n$ or $C_n$. Let $\lam=w\mu$ for $w\in W^J$. Then $p(\lam^-)=p({}^-{w})$.
    
\end{Lem}
\begin{Lem}[{\cite[Lem. 5.3]{BXX}}]
Suppose $\Phi=D_n$. Let $\lam=w\mu$ for $w\in W^J$. Then $p(\lam^-)=p({}^-{w})$ or $p({}^-{(wt)})$.
    
\end{Lem}

\begin{Pro}[{\cite[Thm. 6.4]{BX}}]\label{a-p}
Let $L(\lambda)$ be a Harish-Chandra  module of $SU(p,q)$ with highest weight $\lambda-\rho$ and $\lambda=(\lam_1,\cdots,\lam_n)\in\mathfrak{h}^*$.
\begin{itemize}
\item [(i)] If $ \lambda_1-\lambda_{p+1}\in\mathbb{Z} $, i.e. $ \lambda $ is an integral weight, then $ P(\lambda) $ is a Young tableau with at most two columns. In this case, we will have $ V(L(\lambda))=\overline{\mathcal{O}}_k $, where $ k $ is the number of entries in the second column of $ P(\lambda)$.
\item [(ii)] If $ \lambda_1-\lambda_{p+1}\notin\mathbb{Z} $, i.e. $ \lambda $ is not an integral weight,
we will have $ V(L(\lambda))=\overline{\mathcal{O}}_r $, where $r=\min(p,q)$.
\end{itemize}
\end{Pro}

\begin{Pro}[{\cite[Thm. 6.2]{BXX}}]\label{BCD}
		Let $L(\lambda)$ be a Harish-Chandra  module of $G$ with highest weight $\lambda-\rho$ and $\lambda=(\lam_1,\cdots,\lam_n)\in\mathfrak{h}^*$. Set  $q=q(\lambda^-)=(q_1,q_2,\cdots,q_{2n})=p(\lambda^-)^t$ when $G$ is of type $B, C$ or $D$. Then $V(L(\lambda))=\overline{\mathcal{O}}_{k(\lambda)}$ with $k(\lambda)$ given as follows.
\begin{enumerate}		
 
\item $G=Sp(n, \mathbb{R})$ with $n \geq 2$. Then
$$
k(\lambda)= \begin{cases}2 q_{2}^{\od}, & \text{ if } \lam_{1} \in \mathbb{Z}, \\ 2 q_{2}^{\ev}+1, & \text{ if } \lam_{1} \in \frac{1}{2}+\mathbb{Z}, \\ n, & \text{ otherwise. }\end{cases}
$$

\item $G=S O^{*}(2 n)$ with $n \geq 4$. Then
$$
k(\lambda)= \begin{cases}q_{2}^{\ev}, & \text{ if } \lam_{1} \in \frac{1}{2} \mathbb{Z}, \\ \left\lfloor\frac{n}{2}\right\rfloor, & \text{ otherwise. }\end{cases}
$$
\item $G=S O(2,2 n-1)$ with $n \geq 3$. Then
$$
k(\lambda)= \begin{cases}0, & \text{ if } \lam_{1}-\lam_{2} \in \mathbb{Z}, \lam_{1}>\lam_{2}, \\ 1, & \text{ if } \lam_{1}-\lam_{2} \in \frac{1}{2}+\mathbb{Z}, \lam_{1}>0, \\ 2, & \text{ otherwise. }\end{cases}
$$
\item $G=S O(2,2 n-2)$ with $n \geq 4$. Then
$$
k(\lambda)= \begin{cases}0, & \text{ if } \lam_{1}-\lam_{2} \in \mathbb{Z}, \lam_{1}>\lam_{2}, \\ 1, & \text{ if } \lam_{1}-\lam_{2} \in \mathbb{Z},-\left|\lam_{n}\right|<\lam_{1} \leq \lam_{2}, \\ 2, & \text{ otherwise. }\end{cases}
$$		
\end{enumerate}			
		
\end{Pro}

\section{The case of type \texorpdfstring{$A_{n-1}$}{}}\label{a-thm}

In this section, we assume $\mathfrak{g}=\mathfrak{sl}(n,\mathbb{C})$.





First we recall the famous  Robinson--Schensted  algorithm. Some more details can be found in \cite{BX}.

For  a totally ordered set $ \Gamma $, we  denote by $ \mathrm{Seq}_n (\Gamma)$ the set of sequences $ x=(x_1,x_2,\dots, x_n) $   of length $ n $ with $ x_i\in\Gamma $. In our paper, we usually take $\Gamma$ to be $\mathbb{Z}$ or a coset of $\mathbb{Z}$ in $\mathbb{C}$.
	Then we have a  Young tableau $P(x)$ obtained by applying the following Robinson-Schensted algorithm  to $x\in \mathrm{Seq}_n (\Gamma)$. 
 \begin{definition}
     [Robinson--Schensted  algorithm]
For an element  $ x \in  \mathrm{Seq}_n (\Gamma)$, we write  $x=(x_1,\dots,x_n)$. We associate to $x $ a  Young tableau  $ P(x) $ as follows. Let $ P_0 $ be an empty Young tableau. Assume that we have constructed Young tableau $ P_k $ associated to $ (x_1,\dots,x_k) $, $ 0\leq k<n $. Then $ P_{k+1} $ is obtained by adding $ x_{k+1} $ to $ P_k $ as follows. Firstly we add $ x_{k+1} $ to the first row of $ P_k $ by replacing the leftmost entry $ x_j $ in the first row which is \textit{strictly} bigger than $ x_{k+1} $.  (If there is no such an entry $ x_j $, we just add a box with entry $x_{k+1}  $ to the right side of the first row, and end this process). Then add this $ x_j $ to the next row as the same way of adding $x_{k+1} $ to the first row.  Finally we put $P(x)=P_n$.
 \end{definition}

We use $p(x)=(p_1,\dots, p_k)$ to denote the shape of $P(x)$, where $p_i$ is the number of boxes in the $i$-th row of  $P(x)$.
When $\sum\limits_{1\leq i\leq k} p_i=N$, $p(x)$ will be a partition of $N$ and we still denote this partition by $p(x)=[p_1,\dots, p_k]$.

 For an element  $ w\in\ S_n $, we write  $w=(w_1,...,w_n)$. By using our  RS algorithm, we can associate to $w $ a  Young tableau  $ P(w) $. This is the usual RS algorithm. Some more details can be found in \cite{Sagan}.

\begin{Pro}[ \cite{Ariki2000} or \cite{KL}]
	For $\mathfrak{g}=\mathfrak{sl}(n, \mathbb{C})$ and $ W=S_n$,  two elements $x$ and $ y$ in $S_n$ are in the same KL right cell if and only if $P(x)=P(y)$.
\end{Pro}

  Recall that we use $\mathbf{p}=[p_1,p_2,\dots,p_k]$ to denote a partition of $n$ such that $p_1\geq p_2\geq \dots\geq p_k$ is  a sequence of non-negative integers and $|{\mathbf{p}}|:=p_1+\dots+p_k=n$.  
    A partition $\mathbf{p}$ can be identified with a Young diagram $P$ such that ${p}_i$ is the number of  boxes of the $i$-th row of  the Young diagram $P$.  By abuse of notation, 
    we also use $\mathbf{p}$ to denote the Young diagram corresponding to the partition $\mathbf{p}$.
    As usual ${\bf d}^t=[q_1,q_2,\dots,q_l]$ is the dual partition of $\mathbf{p}$ corresponding to the transpose of the Young diagram $\mathbf{p}$.

From \cite[Thm. 10.1.1]{CM},  we  can identify the set of Young diagrams of total size $n$ with the set of complex nilpotent orbits such that each row corresponds to a Jordan block.  
    Consequently,  the set of Young diagrams is also identified with the set  $\mathrm{Irr}(S_{n})$ via the Springer correspondence. We use $\pi_{\bf p}$ to denote the corresponding irreducible representation of $S_n$ for a given partition $\bf p$.  Note that the trivial orbit ${\bf p}=[1^n]$ (the Young diagram having a single column) corresponds to the sign representation $\mathrm{sgn}$ of  $S_n$. 
 

\begin{Pro}[Steinberg \cite{St}]\label{2.1}
	For any $w,y \in W=S_n$, we have $\mathcal{V}(w)=\mathcal{V}(y)$ if and only if $w \stackrel{R}{\sim} y$.
	
\end{Pro}

\begin{Pro}[Joseph \cite{Jo84}]\label{2.2}
	For $\mathfrak{g}=\mathfrak{sl}(n, \mathbb{C})$ and $ W=S_n$, we have $\mathcal{V}(w)\subseteq  V(L_w)$ and $\dim\mathcal{V}(w)=\dim V(L_w)=\operatorname{GKdim} (L_w)$.
\end{Pro}

\subsection{The case of $\mathfrak{su}(p,q)$}
In this subsection, we assume $\mathfrak{g}=\mathfrak{sl}(n,\mathbb{C})$ and $\mathfrak{g}_{\mathbb{R}}=\mathfrak{su}(p,q)$ with $p+q=n$. 
 A weight $\lambda=(\lambda_1,\lambda_2,...,\lambda_n)$ is called integral if and only if $\lambda_i-\lambda_{i+1}\in\mathbb{Z}$ for $1\leq i\leq n-1$, and called \textit{$ (p,q) $-dominant} if and only if $ \lambda_i-\lambda_j\in\mathbb{Z}_{> 0} $ for $ 1\leq i<j\leq p $ and $ p+1\leq i<j\leq p+q=n $. From Enright--Howe--Wallach \cite{EHW}, we know that the highest weight $SU(p,q)$-module $L(\lambda)$ is a Harish-Chandra module if and only if $\lambda $ is $ (p,q)$-dominant. Therefore, $L_w$ is a highest weight Harish-Chandra $SU(p,q)$-module if and only if $-w\rho$ is $(p,q)$-dominant. The number of these modules is $|\mathcal{W}|=|S_n/(S_p\times S_q)|=\frac{n!}{p!q!}$.

From  \cite{NOT}, for $G=SU(p,q)$ we know that $$\mathcal{O}_j=\{\left(
   \begin{array}{cc}
     0 & c \\
      0 & 0 \\
       \end{array}
       \right)|c\in M_{p\times q}(\mathbb{C}), \text{rank}(c)=j \}$$ 
       is the $j$-th $K_{\mathbb{C}}$-orbit of $G$.
Also we have $$\overline{\mathcal{O}}_m=\coprod\limits_{j\leq m}\mathcal{O}_j=\{\left(
   \begin{array}{cc}
     0 & c \\
      0 & 0 \\
       \end{array}
       \right)|c\in M_{p\times q}(\mathbb{C}), \text{rank}(c)\leq j \}.$$
Sometime we will write it as $\overline{\mathcal{O}}_m(p,q)$ to emphasize that it comes from $SU(p,q)$.

Suppose $s\geq t$. An element $w\in S_{n}$ is called \textit{$(s,t)$-decreasing} if $w=(x_s,...,x_1,y_{t},...,y_1)$ with $x_s>...>x_1$, $y_{t}>...>y_1$ and $x_1<y_1$,...,$x_t<y_t$.

\begin{Thm}\label{a-cell}
Suppose  $L_w$  is a  highest weight module of $\mathfrak{sl}(n,\mathbb{C})$. Then the followings are equivalent:
 \begin{enumerate}
     \item $L_w$ is a Harish-Chandra $SU(p,q)$-module.
     \item $V(L_w)=\mathcal{V}(w)=\overline{\mathcal{O}}_{k}(p,q)$ for some $0\leq k\leq \min\{p,q\}$.
     \item $  w\stackrel{R}{\sim} w_{p,k}=(n,...,p+k+1,p,...,1,p+k,...,p+1).$
 \end{enumerate}

\end{Thm}

\begin{proof} When $L_w$  is a Harish-Chandra $SU(p,q)$-module, we will have $V(L_w)=\mathcal{V}(w)=\overline{\mathcal{O}}_{k}(p,q)$ by Proposition \ref{a-p}
and Proposition \ref{2.2}. Thus (1) $\Rightarrow$ (2).

From Proposition \ref{HC},   $V(L_w)=\mathcal{V}(w)=\overline{\mathcal{O}}_k(p,q)\subset \mathfrak{p}^+$ will imply that $L_w$ is a highest weight Harish-Chandra $SU(p,q)$-module. Thus (2) $\Rightarrow$ (1).

Suppose (1) holds. Now we  prove that  $w$ is right cell equivalent to
\begin{equation*}
  w_{p,k}=(n,n-1,...,p+k+1,p,p-1,...,1,p+k,p+k-1,...,p+1).
\end{equation*}



{\bf Step 1}. We write $\lambda=-w\rho=(\lam_1,...,\lam_p,\lam_{p+1},...,\lam_n)$, which is $(p,q)$-dominant. Let $k$ be the maximum nonnegative integer  for which there exists a sequence of indices $$1\leq i_1<i_2<...<i_k\leq p<p+1\leq j_k<...<j_1\leq n$$ such that $$\lam_{i_1}-\lam_{j_1}\leq 0, \lam_{i_2}-\lam_{j_2}\leq 0,...,\lam_{i_k}-\lam_{j_k}\leq 0.$$ Then from  \cite[Thm. 5.2]{BX}, we can get a Young tableau $P(\lambda)$ with two columns or one column (for $w=-Id$).
We only consider the nontrivial case. Then from  \cite[Lem. 4.5]{BX},  $P(\lambda)$ and $P(w)$ have the same shape. From the construction of $P(\lambda)$, we know that the number of entries in the second column of $P(\lambda)$ equals to $k$. Suppose the entries in the first column of $P(w)$ are $(s_{1},...,s_{n-k})$ and the entries in the second column of $P(w)$ are $(a_{1},...,a_{k})$ from  top to bottom. Then we get $$w\stackrel{R}{\sim}(s_{n-k},...,s_{1}, a_{k},...,a_{1}):=\pi_w,$$ while $\pi_w$ is $(n-k,k)$-decreasing. By Proposition \ref{a-p} and Proposition \ref{2.2},  we have
\begin{equation*}
  V(L_w)=\overline{\mathcal{O}}_{k}(p,q)=\mathcal{V}(w)=\mathcal{V}(\pi_w).
\end{equation*}

{\bf Step 2}. Now we take a special element
\begin{align*}
  \sigma_{p,k}=(p,...,p-k+1,n,...,p+1,p-k,...,1),
\end{align*}
which  is $(k,n-k)$-decreasing. By using the RS algorithm, we can get
\[
	\tiny{\begin{tikzpicture}[scale=1.2,baseline=-55pt]
			\hobox{0}{0}{p-k+1}
			\hobox{0}{1}{\vdots}
			\hobox{0}{2}{p}
	\end{tikzpicture}}\to
	\tiny{\begin{tikzpicture}[scale=1.2,baseline=-55pt]
			\hobox{0}{0}{p-k+1}
			\hobox{0}{1}{\vdots}
			\hobox{0}{2}{p}
			\hobox{1}{0}{n-k+1}
   \hobox{1}{1}{\vdots}
			\hobox{1}{2}{n}
	\end{tikzpicture}}\to
	\tiny{\begin{tikzpicture}[scale=1.2,baseline=-55pt]
			\hobox{0}{0}{p-k+1}
			\hobox{0}{1}{\vdots}
			\hobox{0}{2}{p}
   \hobox{0}{3}{p+k+1}
    \hobox{0}{4}{\vdots}
     \hobox{0}{5}{n}
			\hobox{1}{0}{p+1}
   \hobox{1}{1}{\vdots}
			\hobox{1}{2}{p+k}
	\end{tikzpicture}}\to
	\tiny{\begin{tikzpicture}[scale=1.2,baseline=-55pt]
			\hobox{0}{0}{1}
			\hobox{0}{1}{\vdots}
			\hobox{0}{2}{p}
   \hobox{0}{3}{p+k+1}
    \hobox{0}{4}{\vdots}
     \hobox{0}{5}{n}
			\hobox{1}{0}{p+1}
   \hobox{1}{1}{\vdots}
			\hobox{1}{2}{p+k}
	\end{tikzpicture}}=P(\sigma_{p,k}).
	\]

Note that if we write $-\rho=(b_1,...b_n)$, then
$$-\sigma_{p,k}\rho=(b_n,b_{n-1},...,b_{n-p+k+1},b_{k},b_{k-1},...b_1,b_{n-p+k},...,b_{2k},...b_{k+1})$$ is $(p,q)$-dominant. So $\sigma_{p,k}\in \mathcal{W}$.  Applying the process in {\bf Step 1} to $\sigma_{p,k}$, we can get that 
$\sigma_{p,k}$ is right cell equivalent to $\pi_{\sigma_{p,k}}=w_{p,k}$. From  \cite[Thm. 6.4]{BX} and Proposition \ref{2.2}, we have
\begin{equation*}
  V(L_{\sigma_{p,k}})=\overline{\mathcal{O}}_{k}(p,q)=\mathcal{V}(\sigma_{p,k})=\mathcal{V}(w_{p,k}).
\end{equation*}
Thus $\mathcal{V}(\pi_w)=\mathcal{V}(w_{p,k})$. By Proposition \ref{2.1}, we have $w\stackrel{R}{\sim} \pi_w \stackrel{R}{\sim} w_{p,k}$. Thus (1) $\Rightarrow$ (3).

Now we   suppose $w\stackrel{R}{\sim} w_{p,k}$ for some $0\leq k\leq \min\{p,q\}$. By a direct calculation, we have 
\begin{align*}-w_{p,k}\rho=&
  \frac{1}{2}(n-2k-1,n-2k-3,\cdots,n-2k-2p+1,\\
  &n-1,n-3,\cdots,n-2k+1,n-2k-2p-1,\cdots,1-n).
\end{align*}
It is easy to check that $-w_{p,k}\rho$ is $(p,q)$-dominant and \begin{equation*}
V(L_{w_{p,k}})=\overline{\mathcal{O}}_{k}(p,q)=\mathcal{V}(w_{p,k}).
\end{equation*} 
Thus (3) $\Rightarrow$ (1) and (2).

Now we complete the proof.       

\end{proof}

From Theorem \ref{a-cell}, we can see that  there is only one Harish-Chandra cell (of infinitesimal character $\rho$) with associated variety $\overline{\mathcal{O}}_{k}(p,q)$. This is compatible with Theorem \ref{hermitian av}.


\begin{Cor}\label{cor2}
	
When $G=SU(p,q)$, let $L(\lambda)$  be an integral highest weight Harish-Chandra module. 
 Then we have $V(L(\lambda))=V(L_{w})=\overline{\mathcal{O}}_{k}(p,q)$ if and only if
	\begin{equation*}
	w_{\lambda}\stackrel{R}{\sim}w\stackrel{R}{\sim} w_{p,k}=(n,...,p+k+1,p,...,1,p+k,...,p+1).
	\end{equation*}

\end{Cor}

Note that $w_{p,k}$ is reduced to $(p,p-1,\cdots,1,n,n-1,\cdots,p+1)$ in Corollary \ref{cor2} if $p+k+1>n$ (equivalently $p\geq q=k$). 

\subsection{ The size of highest weight Harish-Chandra cells of type $A_{n-1}$}

Recall the concept of Harish-Chandra cells in Section 2. In the case of type $C$, Barchini--Zierau \cite{BZ} had computed the size of each Harish-Chandra cell consisting of highest weight Harish-Chandra modules with a given associated variety. In the case of type $A$, we will give a similar  result. To prove our result, we recall the famous hook formula. Some details can be found in James \cite{Ja78}

\begin{Pro}[The hook formula]\label{hook}
Let $P$ be a standard Young tableau and $\mathscr{C}$ be the corresponding Kazhdan--Lusztig right cell in the symmetric group $S_n$. Then 
$$\#\mathscr{C}=\frac{n!}{h_1 h_2...h_n},$$
where $h_i$ are the "hooks" of the Young tableau of shape $P$.
\end{Pro}

\begin{Thm}\label{size-a}
We use $\mathcal{C}_k$ to denote a Harish-Chandra cell consisting of  highest weight Harish-Chandra $SU(p,q)$-modules $L_w$ with associated variety $\overline{\mathcal{O}}_{k}(p,q)$, then for $2\leq k\leq\min\{p,q\}$, we have
\begin{align*}
\#\mathcal{C}_k&=\dim\pi_{[2^k,1^{n-k}]}=\text{the number of standard Young tableaux of shape~} P(w)\\
&=\frac{n(n-1)...(n-k+2)(n-2k+1)}{k!}\\
&=\binom{n}{k}-\binom{n}{k-1}.\end{align*}
And $\#\mathcal{C}_0=1$, $\#\mathcal{C}_1=n-1$.
\end{Thm}
\begin{proof}From the proof of Theorem \ref{a-cell}, we know that the number of modules in the Harish-Chandra cell $\mathcal{C}_k$ equals the number of elements in the right cell $\mathscr{C}_k$ which contains $w_{p,k}$. On the other hand, $\#\mathscr{C}_k$ equals the number of the standard Young tableau of shape $p(w_{p,k})=[2^k,1^{n-k}]$, which is just the dimension of $\pi_{p(w_{p,k})}$. We know the numbers of entries in the two columns of $P(w_{p,k})$ are $c_1(P(w_{p,k}))=n-k$ and $c_2(P(w_{p,k}))=k$.  From James \cite{Ja78}, the corresponding  hooks of $P(w_{p,k})$ is a Young tableau $P_k$ consisting of two columns. The entries in the first column of $P_k$ are $(n-(k-1),n-k,...,n-2k+2,n-2k,...,1)$ and the entries in the second column of $P_k$ are $(k,k-1,...,1)$ from  top to bottom. From Proposition \ref{hook}, we have
\begin{align*}\#\mathscr{C}_k&=\frac{n!}{(n-(k-1))\times...\times(n-2k+2)\times (n-2k)\times...\times1\times k\times...\times 1}\\
&=\frac{n(n-1)...(n-k+2)(n-2k+1)}{k!}\\
&=\binom{n}{k} -\binom{n}{k-1}.
\end{align*}

When $k=1$, we will have $c_1(P(w_{p,k}))=n-1$ and $c_2(P(w_{p,k}))=1$, so $$\#\mathcal{C}_1=\frac{n!}{n\times(n-2)\times...\times1\times1}=n-1.$$

When $k=0$, we will have $c_1(P(w_{p,k}))=n$ and $c_2(P(w_{p,k}))=0$, so $\#\mathcal{C}_0=1$.

\end{proof}

If we set $\binom{n}{0}=1$ and $\binom{n}{-1}=0$, our formula in Theorem \ref{size-a} will be uniform.

\section{The case of type \texorpdfstring{$D_n$}{}}
In this section, we assume $\mathfrak{g}=\mathfrak{so}(2n,\mathbb{C})$ and $\mathfrak{g}_{\mathbb{R}}=\mathfrak{so}^*(2n)$ or $\mathfrak{so}(2,2n-2)$. We  use $W'_n$ to denote the Weyl group of $\mathfrak{g}=\mathfrak{so}(2n,\mathbb{C})$.

\subsection{The case of $\mathfrak{so}^*(2n)$}

From Enright--Howe--Wallach \cite{EHW}, we know that the highest weight module $L(\lambda)$ is an integral  Harish-Chandra module if and only if  $\lam_i-\lam_{i+1}\in \mathbb{Z}_{> 0}$ for $1\leq i\leq n-1$ and $2\lam_k\in \mathbb{Z}$ for $1\leq k\leq n$. Therefore, $L_w$ is a highest weight Harish-Chandra $SO^*(2n)$-module if and only if $-w\rho=(\lam_1,\lam_2,\cdots,\lam_n)$ with  $\lam_i-\lam_{i+1}\in \mathbb{Z}_{ >0}$ for $1\leq i\leq n-1$. The number of these modules is $|\mathcal{W}|=|W'_n/S_{n}|=\frac{n!2^{n-1}}{n!}=2^{n-1}$.

First we consider the Wallach representations.

\begin{Thm}\label{d-wlam}
Suppose $\frg_{\mathbb{R}}=\frso^*(2n)$ and $L(\lambda)$ is a highest weight Harish-Chandra module with highest weight $\lambda-\rho =-k c\zeta $. Suppose $w_{\lambda}\in W$ is the minimal length element such that $\mu=w_{\lambda}^{-1}\lambda$ is antidominant, then we have
\begin{enumerate}
    \item if $k\in \mathbb{Z},1\leq k\leq [\frac{n}{2}]$ and $2k\leq n-1 $, then 
$$w_\lambda=\begin{cases}
    (k+1,-(k+2),k, -(k+3),k-1,\dots,&\\
    -2k,2,-(2k+1),1,-(2k+2),\dots, -n), &\text{~if~} n-k \text{~is~ odd},\\
     (k+1,-(k+2),k, -(k+3),k-1,\dots,&\\
     -2k,2,-(2k+1),-1,-(2k+2),\dots, -n), &\text{~if~} n-k \text{~is~ even}.
\end{cases}$$ 
\item if $n$ is even and $k=\frac{n}{2}$, then 
$$w_\lambda=\begin{cases}
    (\frac{n}{2}+1,-(\frac{n}{2}+2),\frac{n}{2},-(\frac{n}{2}+3),\frac{n}{2}-1,\dots,-n,2,1), &\text{~if~} \frac{n}{2}\text{~is~ odd},\\
     (\frac{n}{2}+1,-(\frac{n}{2}+2),\frac{n}{2},-(\frac{n}{2}+3),\frac{n}{2}-1,\dots,-n,2,-1), &\text{~if~} \frac{n}{2} \text{~is~ even}.
\end{cases}$$
\end{enumerate}


\end{Thm}
\begin{proof}
    From \cite{EHW}, we have $\lambda=-k c\zeta +\rho=(n-1-k,n-2-k,\dots, 1-k,-k)$.
By Proposition \ref{BCD}, we have $V(L(\lambda))=\overline{\mathcal{O}}_{k}.$ 

Now we suppose $n-1-k\geq k$. By Proposition  \ref{anti},  we have $\mu=(-(n-1-k),\dots,-(k+1),-k,-k,1-k,1-k,\dots,-1, -1,0)$. Thus we have:
  $$\tiny{\begin{tikzpicture}
     \draw(-1,-4) node{$ \mu^-: $};
      \draw (-1,0) node{$ \lambda^-: $};
    \node (num1) at (0,0) {$n-1-k$};
    \node (num2) at (1,0) {$\dots$};
    \node (num3) at (1.8,0) {$1$};
    \node (num4) at (2.7,0) {$0$};
    \node (num5) at (3.5,0) {$-1$};
    \node (num6) at (4.3,0) {$\dots$};
    \node (num7) at (5.1,0) {$-k$};
    \node (num8) at (6.1,0) {$k$};
    \node (num9) at (7.0,0) {$\dots$};
    \node (num10) at (7.8,0) {$1$};
    \node (num11) at (8.6,0) {$0$};
    \node (num12) at (9.4,0) {$-1$};
    \node (num13) at (10.2,0) {$\dots$};
    \node (num14) at (11,0) {$1+k-n$};

    \node (num1') at (0,-4) {$1+k-n$};
    \node (num2') at (1,-4) {$\dots$};
    \node (num3') at (1.8,-4) {$-k$};
    \node (num4') at (2.7,-4) {$\dots$};
    \node (num5') at (3.5,-4) {$-1$};
    \node (num6') at (4.3,-4) {$-1$};
    \node (num7') at (5.1,-4) {$0$};
    \node (num8') at (6.1,-4) {$0$};
    \node (num9') at (7.0,-4) {$1$};
    \node (num10') at (7.8,-4) {$1$};
    \node (num11') at (8.6,-4) {$\dots$};
    \node (num12') at (9.4,-4) {$k$};
    \node (num13') at (10.2,-4) {$\dots$};
    \node (num14') at (11,-4) {$n-1-k$};
    
    \draw(num1') -- (num14);
    \draw(num3') -- (num7);
   \draw(num5') -- (num5);
    \draw (num6') -- (num12);
   \draw(num7') -- (num4);
   \draw(num8') -- (num11); 
   \draw(num9') -- (num3); 
   \draw(num10') -- (num10); 
   \draw(num12') -- (num8);    
   \draw(num14') -- (num1);

    \draw[dashed](5.6,0.3)--(5.6,-4.3);
\end{tikzpicture}}$$

By Lemma \ref{lambda-w}, we can find that $$w_\lambda=\begin{cases}
    (k+1,-(k+2),k, -(k+3),k-1,\dots,&\\
    -2k,2,-(2k+1),1,-(2k+2),\dots, -n), &\text{~if~} n-k \text{~is~ odd},\\
     (k+1,-(k+2),k, -(k+3),k-1,\dots,&\\
     -2k,2,-(2k+1),-1,-(2k+2),\dots, -n), &\text{~if~} n-k \text{~is~ even}.
\end{cases}$$  
 By using the RS algorithm, when $ n-k$ is odd,  we will have
 \[P(^{-}w_{\lambda})=\tiny{\begin{tikzpicture}[scale=0.8,baseline=-80pt]
			\hobox{0}{0}{-n}
			\hobox{0}{1}{-n+1}
			\hobox{0}{2}{\vdots}
                \hobox{0}{3}{-1}
               \hobox{0}{4}{2k+2}
                \hobox{0}{5}{\vdots}
                \hobox{0}{6}{n}
			\hobox{1}{0}{1}
			\hobox{1}{1}{\vdots}
                 \hobox{1}{2}{2k+1}
   \end{tikzpicture}} .\]

By using the RS algorithm, when $ n-k$ is even,  we will have
\[P(^{-}w_{\lambda})=\tiny{\begin{tikzpicture}[scale=0.8,baseline=-90pt]
			\hobox{0}{0}{-n}
			\hobox{0}{1}{-n+1}
			\hobox{0}{2}{\vdots}
                \hobox{0}{3}{-2}
                \hobox{0}{4}{1}
               \hobox{0}{5}{2k+2}
                \hobox{0}{6}{\vdots}
                \hobox{0}{7}{n}
			\hobox{1}{0}{-1}
   \hobox{1}{1}{2}
			\hobox{1}{2}{\vdots}
                 \hobox{1}{3}{2k+1}
   \end{tikzpicture}} .\]

If $n-1-k<k\leq [\frac{n}{2}]$,   $n$ will be even and $k=\frac{n}{2}$. Then we have $\lambda=-k c\zeta +\rho=(\frac{n}{2}-1,\frac{n}{2}-2,\dots, 1,0,-1,\dots ,-\frac{n}{2}) $. By Proposition  \ref{anti},  we have $\mu=(-\frac{n}{2},-(\frac{n}{2}-1),-(\frac{n}{2}-1),\dots,-1, -1,0)$. Thus we have:
 $$\tiny{\begin{tikzpicture}
     \draw(0,-4) node{$ \mu^-: $};
      \draw (0,0) node{$ \lambda^-: $};
    \node (num1) at (0.8,0) {$\frac{n}{2}-1$};
   \node (num2) at (1.6,0) {$\dots$};
    \node (num3) at (2.4,0) {$1$};
    \node (num4) at (3.2,0) {$0$};
    \node (num5) at (4,0) {$-1$};
    \node (num6) at (4.8,0) {$\dots$};
    \node (num7) at (5.6,0) {$1-\frac{n}{2}$};
    \node (num8) at (6.4,0) {$-\frac{n}{2}$};
    \node (num9) at (7.4,0) {$\frac{n}{2}$};
    \node (num10) at (8.2,0) {$\frac{n}{2}-1$};
    \node (num11) at (9,0) {$\dots$};
    \node (num12) at (9.8,0) {$1$};
    \node (num13) at (10.6,0) {$0$};
    \node (num14) at (11.4,0) {$-1$};
    \node (num15) at (12.2,0) {$\dots$};
    \node (num16) at (13,0) {$1-\frac{n}{2}$};

    \node (num1') at (0.8,-4) {$-\frac{n}{2}$};
    \node (num2') at (1.6,-4) {$1-\frac{n}{2}$};
    \node (num3') at (2.4,-4) {$1-\frac{n}{2}$};
    \node (num4') at (3.2,-4) {$\dots$};
    \node (num5') at (4,-4) {$\dots$};
    \node (num6') at (4.8,-4) {$-1$};
    \node (num7') at (5.6,-4) {$-1$};
    \node (num8') at (6.4,-4) {$0$};
    \node (num9') at (7.4,-4) {$0$};
    \node (num10') at (8.2,-4) {$1$};
    \node (num11') at (9,-4) {$1$};
    \node (num12') at (9.8,-4) {$\dots$};
    \node (num13') at (10.6,-4) {$\dots$};
    \node (num14') at (11.4,-4) {$\frac{n}{2}$-1};
    \node (num15') at (12.2,-4) {$\frac{n}{2}$-1};
    \node (num16') at (13,-4) {$\frac{n}{2}$};

    \draw(num1') -- (num8);
    \draw(num2') -- (num7);
   \draw(num3') -- (num16);
    \draw (num6') -- (num5);
   \draw(num7') -- (num14);
   \draw(num8') -- (num4); 
   \draw(num9') -- (num13); 
   \draw(num10') -- (num3); 
   \draw(num11') -- (num12);    
   \draw(num14') -- (num1);
   \draw(num15') -- (num10);
   \draw(num16') -- (num9);

    \draw[dashed](6.9,0.3)--(6.9,-4.3);
\end{tikzpicture}}$$

By Lemma  \ref{lambda-w}, we can find that 
$$w_\lambda=\begin{cases}
    (\frac{n}{2}+1,-(\frac{n}{2}+2),\frac{n}{2},-(\frac{n}{2}+3),\frac{n}{2}-1,\dots,-n,2,1), &\text{~if~} \frac{n}{2}\text{~is~ odd},\\
     (\frac{n}{2}+1,-(\frac{n}{2}+2),\frac{n}{2},-(\frac{n}{2}+3),\frac{n}{2}-1,\dots,-n,2,-1), &\text{~if~} \frac{n}{2} \text{~is~ even}.
\end{cases}$$

By using the RS algorithm, when $\frac{n}{2}$ is odd, we will have
\[P(^{-}w_{\lambda})=\tiny{\begin{tikzpicture}[scale=0.8,baseline=-35pt]
			\hobox{0}{0}{-n}
			\hobox{0}{1}{\vdots}
                \hobox{0}{2}{-1}
			\hobox{1}{0}{1}
			\hobox{1}{1}{\vdots}
                 \hobox{1}{2}{n}
   \end{tikzpicture}} .\]

By using the RS algorithm, when $\frac{n}{2}$ is even, we will have
\[P(^{-}w_{\lambda})=\tiny{\begin{tikzpicture}[scale=0.8,baseline=-45pt]
			\hobox{0}{0}{-n}
			\hobox{0}{1}{-n+1}
			\hobox{0}{2}{\vdots}
                \hobox{0}{3}{-1}
			\hobox{1}{0}{-1}
   \hobox{1}{1}{2}
			\hobox{1}{2}{\vdots}
                 \hobox{1}{3}{n}
   \end{tikzpicture}} .\]

\end{proof}

Note that when $n$ is even and $k=\frac{n}{2}$ is odd, we may regard that $w_\lambda=(k+1,-(k+2),k, -(k+3),k-1,\dots ,-2k,2,-(2k+1),1,-(2k+2),\dots, -n)$ is reduced to $(\frac{n}{2}+1,-(\frac{n}{2}+2),\frac{n}{2},-(\frac{n}{2}+3),\frac{n}{2}-1,\dots,-n,2,1)$. When $n$ is even and $k=\frac{n}{2}$ is even, the reduction is similar.

By using Theorem \ref{hermitian av}, we have the following result.

\begin{Cor}\label{d-cell-so*}
  When $G=SO^*(2n)$, let $L(\lambda)$  be an integral highest weight Harish-Chandra module. 
 Then we have $V(L(\lambda))=V(L_{w})=\overline{\mathcal{O}}_{k}$ for some $k\in \mathbb{Z}$ and $1\leq k\leq [\frac{n}{2}]$ if and only if
	 $$w_\lambda\stackrel{R}{\sim}w\stackrel{R}{\sim}\begin{cases}
    (k+1,-(k+2),k, -(k+3),k-1,\dots,&\\
    -2k,2,-(2k+1),1,-(2k+2),\dots, -n), &\text{~if~} n-k \text{~is~ odd},\\
     (k+1,-(k+2),k, -(k+3),k-1,\dots,&\\
     -2k,2,-(2k+1),-1,-(2k+2),\dots, -n), &\text{~if~} n-k \text{~is~ even}.
\end{cases}
$$


\end{Cor}

\subsection{The case of $\mathfrak{so}(2,2n-2)$}
From Enright--Howe--Wallach \cite{EHW}, we know that the highest weight module $L(\lambda)$ is an integral  Harish-Chandra module if and only if $\lam_1\pm\lam_2\in \mathbb{Z}$, $\lam_i-\lam_{i+1}\in \mathbb{Z}_{> 0}$ for $2\leq i\leq n-1$ and $\lam_{n-1}-|\lam_n|>0$. Therefore, $L_w$ is a highest weight Harish-Chandra $SO(2,2n-2)$-module if and only if $-w\rho=(\lam_1,\lam_2,\cdots,\lam_n)$ with  $\lam_i-\lam_{i+1}\in \mathbb{Z}_{ >0}$ for $2\leq i\leq n-1$ and  $\lam_{n-1}-|\lam_n|>0$. The number of these modules is $|\mathcal{W}|=|W'_n/W'_{n-1}|=\frac{n!2^{n-1}}{(n-1)!2^{n-2}}=2n$.

First we consider the Wallach representations.

\begin{Thm}\label{wlam-2}
Suppose $\frg_{\mathbb{R}}=\frso(2,2n-2)$ and $L(\lambda)$ is a highest weight Harish-Chandra module with highest weight $\lambda-\rho =-k c\zeta $. Suppose $w_{\lambda}\in W$ is the minimal length element such that $\mu=w_{\lambda}^{-1}\lambda$ is antidominant, then we have
$$w_{\lambda}=
\begin{cases}
(1,-n,-2,\dots ,-(n-1)), & \text{if~} k=1 \text{~and~} n \text{~is~odd}, \\
(-1,-n,-2,\dots ,-(n-1)), & \text{if~} k=1 \text{~and~} n \text{~is~even}, \\
(1,-2, \dots ,-(n-2),n,-(n-1)),& \text{if~} k=2
\text{~and~} n \text{~is~even},\\ 
(-1,-2, \dots ,-(n-2),n,-(n-1)),& \text{if~} k=2
\text{~and~} n \text{~is~odd}.
\end{cases}
$$


\end{Thm}
\begin{proof}
    From \cite{EHW}, we have $\lambda=-k c\zeta +\rho=(-k(n-2)+(n-1),n-2,n-3,\dots, 1,0) $.
By Proposition \ref{BCD}, we have $V(L(\lambda))=\overline{\mathcal{O}}_{k}.$ 

When $k=1$, $\lambda=(1,n-2,n-3,\dots, 1,0) $. By Proposition \ref{anti},  we have 
$\mu=(-(n-2),-(n-1),\dots,-2,-1 -1,0)$.
 Thus we have:
  $$\begin{tikzpicture}
     \draw(-0.5,-4) node{$ \mu^-: $};
      \draw (-0.5,0) node{$ \lambda^-: $};
    \node (num1) at (0.5,0) {1};
    \node (num2) at (1.4,0) {$n-2$};
    \node (num3) at (2.4,0) {$n-1$};
    \node (num4) at (3.2,0) {$\dots$};
    \node (num5) at (4.1,0) {$1$};
    \node (num6) at (5,0) {$0$};
    \node (num7) at (6,0) {$0$};
    \node (num8) at (6.9,0) {$-1$};
    \node (num9) at (7.8,0) {$\dots$};
    \node (num10) at (8.7,0) {$1-n$};
    \node (num11) at (9.7,0) {$2-n$};
    \node (num12) at (10.6,0) {$-1$};
    
    \node (num1') at (0.5,-4) {$2-n$};
    \node (num2') at (1.5,-4) {$1-n$};
    \node (num3') at (2.3,-4) {$\dots$};
    \node (num4') at (3.2,-4) {$-1$};
    \node (num5') at (4.1,-4) {$-1$};
    \node (num6') at (5,-4) {$0$};
    \node (num7') at (6,-4) {$0$};
    \node (num8') at (6.9,-4) {$1$};
    \node (num9') at (7.8,-4) {$1$};
    \node (num10') at (8.7,-4) {$\dots$};
    \node (num11') at (9.6,-4) {$n-1$};
    \node (num12') at (10.6,-4) {$n-2$};
    
    \draw(num1') -- (num11);
    \draw(num2') -- (num10);
    \draw(num4') -- (num8);
    \draw (num5') -- (num12);
   \draw(num6') -- (num6);
    \draw(num7') -- (num7);    
    \draw(num8') -- (num1);    
    \draw(num9') -- (num5); 
    \draw(num11') -- (num3);
    \draw(num12') -- (num2);
    
    \draw[dashed](5.5,0.3)--(5.5,-4.3);
\end{tikzpicture}$$
By Lemma  \ref{lambda-w}, we can find that $$w_\lambda =\begin{cases}
(1,-n,-2,\dots,-(n-1)), & \text{if~}  n \text{~is~odd}, \\
(-1,-n,-2,\dots,-(n-1)), & \text{if~}  n \text{~is~even}.
\end{cases}$$

When $k=2$, $\lambda=(3-n,n-2,n-3,\dots,2 ,1,0)$. By Proposition \ref{anti},  we have 
$\mu=(-(n-2),-(n-3), -(n-3),-(n-4),\dots, -2,-1,0)$.
 Thus we have:
  $$\tiny{\begin{tikzpicture}
     \draw(-1,-4) node{$ \mu^-: $};
      \draw (-1,0) node{$ \lambda^-: $};
    \node (num1) at (0,0) {$3-n$};
    \node (num2) at (1,0) {$n-2$};
    \node (num3) at (2,0) {$n-3$};
    \node (num4) at (2.8,0) {$\dots$};
    \node (num5) at (3.6,0) {$2$};
    \node (num6) at (4.4,0) {$1$};
    \node (num7) at (5.2,0) {$0$};
    \node (num8) at (6.2,0) {$0$};
    \node (num9) at (7,0) {$-1$};
    \node (num10) at (7.8,0) {$-2$};
    \node (num11) at (8.6,0) {$\dots$};
    \node (num12) at (9.4,0) {$3-n$};
    \node (num13) at (10.4,0) {$2-n$};
    \node (num14) at (11.4,0) {$n-3$};

    \node (num1') at (0,-4) {$2-n$};
    \node (num2') at (1,-4) {$3-n$};
    \node (num3') at (2,-4) {$3-n$};
    \node (num4') at (2.8,-4) {$\dots$};
    \node (num5') at (3.6,-4) {$-2$};
    \node (num6') at (4.4,-4) {$-1$};
    \node (num7') at (5.2,-4) {$0$};
    \node (num8') at (6.2,-4) {$0$};
    \node (num9') at (7,-4) {$1$};
    \node (num10') at (7.8,-4) {$2$};
    \node (num11') at (8.6,-4) {$\dots$};
    \node (num12') at (9.4,-4) {$n-3$};
    \node (num13') at (10.4,-4) {$n-3$};
    \node (num14') at (11.4,-4) {$n-2$};
    
    \draw(num1') -- (num13);
    \draw(num2') -- (num1);
    \draw(num3') -- (num12);
    \draw (num5') -- (num10);
   \draw(num6') -- (num9);
    \draw(num7') -- (num7);    
    \draw(num8') -- (num8);    
    \draw(num9') -- (num6);
    \draw(num10') -- (num5);
    \draw(num12') -- (num3);
    \draw(num13') -- (num14);
    \draw(num14') -- (num2);

    \draw[dashed](5.7,0)--(5.7,-4);
\end{tikzpicture}}$$  
By Lemma  \ref{lambda-w}, we can find that $$w_\lambda =\begin{cases}
(1,-2, \dots ,-(n-2),n,-(n-1)),& \text{if~}  n \text{~is~even},\\ 
(-1,-2, \dots ,-(n-2),n,-(n-1)),& \text{if~}  n \text{~is~odd}.
\end{cases}$$

\end{proof}

By using Theorem \ref{hermitian av}, we have the following result.

\begin{Cor}\label{d-cell}
  When $G=SO(2,2n-2)$, let $L(\lambda)$  be an integral highest weight Harish-Chandra module. 
 Then we have $V(L(\lambda))=V(L_{w})=\overline{\mathcal{O}}_{1}$  if and only if
	 $$w_\lambda\stackrel{R}{\sim}w\stackrel{R}{\sim}\begin{cases}
(1,-n,-2,\dots,-(n-1)), & \text{if~}  n \text{~is~odd}, \\
(-1,-n,-2,\dots,-(n-1)), & \text{if~}  n \text{~is~even},
\end{cases}  
$$
and $V(L(\lambda))=V(L_{w})=\overline{\mathcal{O}}_{2}$  if and only if
	 $$w_\lambda\stackrel{R}{\sim}w\stackrel{R}{\sim}\begin{cases}
(1,-2, \dots ,-(n-2),n,-(n-1)),& \text{if~}  n \text{~is~even},\\ 
(-1,-2, \dots ,-(n-2),n,-(n-1)),& \text{if~}  n \text{~is~odd}.
\end{cases}
$$

\end{Cor}

\subsection{ The size of highest weight Harish-Chandra cells of type $D_n$}
First we recall Lusztig's $D$-symbols \cite{lusztig1977symbol}. Let
	\[ \begin{pmatrix}
	\lambda_1~\lambda_2~\cdots~\lambda_{m}\\\mu_1~\mu_2~\cdots~\mu_m
	\end{pmatrix}, m\geq 0\]
	be a tableau of nonnegative integers such that entries in each row are  strictly increasing. Define an equivalence relation on the set of all such tableaux via
	\begin{equation*}\label{eq:equiv-d}
	\begin{pmatrix}
	\mu_1~\mu_2~\cdots\mu_m\\
	\lambda_1~\lambda_2~\cdots~\lambda_{m}
	\end{pmatrix} \sim
	\begin{pmatrix}
	\lambda_1~\lambda_2~\cdots~\lambda_{m}\\\mu_1~\mu_2~\cdots\mu_m
	\end{pmatrix}
	\sim
	\begin{pmatrix}
	0~\lambda_1+1~\lambda_2+1~\cdots~\lambda_{m}+1\\0~\mu_1+1~\mu_2+1~\cdots\mu_m+1
	\end{pmatrix} .
	\end{equation*}
	Denote by $ \Sigma_D $ the set of equivalence classes under this relation $ \sim $. Use the same notation $ \Lambda=\begin{pmatrix}
	\lambda_1~\lambda_2~\cdots~\lambda_{m}\\\mu_1~\mu_2~\cdots\mu_m
	\end{pmatrix}  \in\Sigma_D$ to denote its equivalence class,  called a \textit{$ D$-symbol}.

Suppose $\mathfrak{g}$ is of type $D_n$. Let ${\mathcal{O}}_{\bf{d}}$ be a nilpotent orbit of $\mathfrak{g}$ with partition ${\bf d}=[d_1,\dots,d_k]$ of $2n$, then $k$ is even. From \cite[\S 10.1]{CM}, 
 there is a $D$-symbol
	$\Lambda =\begin{pmatrix}
	\lambda_1~\lambda_2~\cdots~\lambda_{m}\\\mu_1~\mu_2~\cdots\mu_m
	\end{pmatrix} $ such that, as multisets
	\begin{equation*}\label{eq:symb}
	\{2\lambda_i+1,2 \mu_j\mid  i\leq  m, j\leq m  \} =\{ d_l+k-l \mid  l\leq  k\},
	\end{equation*}
	where $2m=k$.

Recall that the elements of $\mathrm{Irr}(W'_n)$ are parameterized by the unordered bipartitions $\{ \bf{ d},\bf{f}\}$  with $|{\bf d}|+|{\bf f}|=n$, except that if $n=2m$ is even and $\bf{ d}=\bf{f}$, then the unordered bipartition  $\{ \bf{ d},\bf{d}\}$ corresponds to two representations,  $\pi^I_{\{\bf{d},\bf{d}\}}$ and $\pi^{II}_{\{\bf{d},\bf{d}\}}$.	So there is a one-to-one correspondence between D-symbols and  unordered bipartitions 	except that if $n=2m$ is even and $\bf{ d}=\bf{f}$. It can be described as follows. Let $\Lambda =\begin{pmatrix}
	\lambda_1~\lambda_2~\cdots~\lambda_{m}\\\mu_1~\mu_2~\cdots\mu_m
	\end{pmatrix} \in\Sigma_D $. We subtract $i-1$ from the $i$-th element  of the top row and  bottom row.
	Then we get an unordered bipartition  $\{ \bf{ d},\bf{f}\}$ corresponding to the new top row and bottom row.

\begin{Pro}[{\cite[Thm. 10.1.2 and Thm. 10.1.3]{CM}}]\label{d-repsize}
Suppose $\mathfrak{g}$ is of type $D_n$. Then we have
    $$\dim \pi_{\{{\bf d},{\bf f}\}}=\begin{cases}
       \binom{n}{|\bf{d}|} \dim \pi_{\bf d}\cdot \dim \pi_{\bf f}, & \text{if~} {\bf d}\neq {\bf f}\\
         \frac{1}{2}\binom{n}{|\bf{d}|}(\dim \pi_{\bf d})^2, &\text{if~} {\bf d}= {\bf f}.
    \end{cases}$$

\end{Pro}

\begin{Thm}\label{size-d}
We use $\mathcal{C}_k$ to denote a Harish-Chandra cell consisting of  highest weight Harish-Chandra $SO^*(2n)$-modules $L_w$ with associated variety $\overline{\mathcal{O}}_{k}$ for some $k\in \mathbb{Z}$ and $0\leq k\leq [\frac{n}{2}]$. Then we have
$$
\#\mathcal{C}_k=\begin{cases}\binom{n}{k}, & \text{ if } k\neq \frac{n}{2}, \\ 
\frac{1}{2}\binom{n}{\frac{n}{2}}, & \text{ if } k=\frac{n}{2}.\end{cases}$$
\end{Thm}
\begin{proof}
   From \cite[\S 3]{BZ}, we know that there is an irreducible $W'_n$-module $\pi(\caO_k^\C,1)$  for 
 each complex nilpotent $G$-orbit $\mathcal{O}_k^{\mathbb{C}}$ via the Springer correspondence. Suppose $k>0$. From \cite{BJ}, we know that the partition corresponding to $\mathcal{O}_k^{\mathbb{C}}$ is ${\bf p}=[2^{2k},1^{2n-4k}]$. The corresponding $D$-symbol is $$ \Lambda=\begin{pmatrix}
	0~\cdots~n-2k-1~n-2k+1~\cdots~n-k\\
 1~~~\quad~~~\cdots~~\quad\quad~~~n-2k~n-2k+1~\cdots~n-k
	\end{pmatrix},$$
 and the corresponding bipartition is 
 $$\{ {\bf d},{\bf f}\}=\{[1^k],[1^{n-k}]\}.$$
    
By Proposition \ref{hook}, we have $\dim \pi_{\bf d}=\dim \pi_{\bf f}=1$.  Note that when ${\bf d}={\bf f}$, we will have $|{\bf d}|=|{\bf f}|=\frac{n}{2}=k$. 
Thus by Proposition \ref{d-repsize}, we can get 
$$
\#\mathcal{C}_k=\begin{cases}\binom{n}{k}, & \text{ if } k\neq \frac{n}{2}, \\ 
\frac{1}{2}\binom{n}{\frac{n}{2}}, & \text{ if } k=\frac{n}{2}.\end{cases}$$
When $k=0$, we have $\#\mathcal{C}_0=1$ since $\mathcal{C}_0=\{\mathbb{C}\}$. 
\end{proof}

\begin{Thm}\label{size-d2} 
We use $\mathcal{C}_k$ to denote a Harish-Chandra cell consisting of  highest weight Harish-Chandra $SO(2,2n-2)$-modules $L_w$ with associated variety $\overline{\mathcal{O}}_{k}$ for $0\leq k\leq 2$. Then we have
$$
\#\mathcal{C}_2=n-1, \#\mathcal{C}_1=n \text{~and ~}\#\mathcal{C}_0=1.$$
\end{Thm}

\begin{proof}
   The argument is similar to the proof of Theorem \ref{size-d}. From \cite{BJ}, we know that the partition corresponding to $\mathcal{O}_2^{\mathbb{C}}$ is ${\bf p}=[3,1^{2n-3}]$. The corresponding $D$-symbol is $$ \Lambda=\begin{pmatrix}
	0~\cdots~n-3~n-2\\
 1~\quad~\cdots~\quad~n-2~\quad~n
	\end{pmatrix},$$
 and the corresponding bipartition is 
 $$\{ {\bf d},{\bf f}\}=\{\emptyset,[2,1^{n-2}]\}.$$
    
By Proposition \ref{hook}, we have $\dim \pi_{\bf d}=1$ and $\dim \pi_{\bf f}=n-1$. 
Thus by Proposition \ref{d-repsize}, we can get 
$$
\#\mathcal{C}_2=\binom{n}{0}\dim \pi_{\bf d}\cdot \dim \pi_{\bf f}=n-1.$$ 

From \cite{BJ}, we know that the partition corresponding to $\mathcal{O}_1^{\mathbb{C}}$ is ${\bf p}=[2,2,1^{2n-4}]$. The corresponding $D$-symbol is $$ \Lambda=\begin{pmatrix}
	0~\cdots~n-3~n-1\\
 1~~\cdots~n-2~n-1
	\end{pmatrix},$$
 and the corresponding bipartition is 
 $$\{ {\bf d},{\bf f}\}=\{[1],[1^{n-1}]\}.$$
By Proposition \ref{hook}, we have $\dim \pi_{\bf d}=1$ and $\dim \pi_{\bf f}=1$. 
Thus by Proposition \ref{d-repsize}, we can get  $\#\mathcal{C}_1=n$. Finally, we have $\#\mathcal{C}_0=1$.

\end{proof}

\begin{Rem}\label{pycox}
By using PyCox \cite{ge}, we can easily find out all the elements in a given Harish-Chandra cell $\mathcal{C}_k$.
    In PyCox \cite{ge}, the order ``klcellrepelm$(W, w)$" will give us the index $i_w$ of the left cell $\mathcal{C}_L(w)$ containing  $w\in W'_n$. Then the order ``$\mathrm{left}_{-}\mathrm{cell}_{-}\mathrm{infos}[i_w]$" will given us all the elements contained in this left cell $\mathcal{C}_L(w)$. If we choose $w=w_{\lam}^{-1}$ for some $w_{\lam}$ in Theorem \ref{d-wlam} and Theorem \ref{wlam-2}, then by using PyCox, we can find out all the elements $\{x_1,\dots,x_m\}$ contained in the same left cell  $\mathcal{C}_L(w_{\lam}^{-1})$.
    Thus the right cell $\mathcal{C}_R(w^{-1})=\mathcal{C}_R(w_{\lam})$ consists of the following elements:
    $$\{x^{-1}_1,\dots,x^{-1}_m\}.$$
    
\end{Rem}

\section{The cases of type \texorpdfstring{$E_6$}{} and type \texorpdfstring{$E_7$}{}}
In this section, we consider highest weight Harish-Chandra modules of type $E$. We use $W(E_n)$ to denote the Weyl group of $E_n$.

First we recall the method of finding out the minimal length element $w_{\lambda}$ for an integral weight $\lambda$ such that $w_{\lam}^{-1}\lam=\mu$  is antidominant.

 Let $\Delta=\{\alpha_i \mid  1\leq i\leq n\}$ be the simple roots of an exceptional type root system $\Phi$ with corresponding fundamental weights $\{\omega_i\mid 1\leq i\leq n\}$. 
The Cartan matrix is defined by
$A=(A_{ij})_{n\times n}$ with $A_{ij}=(\alpha_i, \alpha_j^\vee)$ and $A_{ij}\in \{0,-1,-2,-3\}$ for $i \neq j$.
 Then  $\alpha_j=\sum\limits_{k=1}^n A_{jk}\omega_k$. For any $\lam\in \mathfrak{h}^*$, we can write $\lam=\sum\limits_{1\leq i\leq n}k_i\omega_i$, where $k_i=(\lam, \alpha_i^{\vee})$. Then an integral weight $\lam$ is  antidominant  if and only if   $k_i\in \mathbb{Z}_{\leq 0}$ for  all $1\leq i\leq n$.

Now we want to find  $w_{\lam}$ for an integral weight $\lambda$. When $\lam$ is antidominant (resp. dominant), $w_{\lam}=\mathrm{Id}$ (resp. $-\mathrm{Id}$). Suppose $\lam$ is not antidominant, then there exist some $k_i\in \mathbb{Z}_{>0}$. We choose the largest index  $i_{\lam}={i_1}$ such that $k_{i_1}\in \mathbb{Z}_{>0}$. Denote the simple reflection $s_{\alpha_i}$ by $s_i$. Then we have $$s_{i_1}\lam=\lam-k_{i_1}\alpha_{i_1}=\lam-k_{i_1}\sum\limits_{j=1}^n A_{i_1j}\omega_j=\lam-2k_{i_1}\omega_{i_1}-k_{i_1}\sum\limits_{j=1,j\neq i_1}^n A_{i_1j}\omega_j.$$
 Then the coefficient of $\omega_{i_1}$ in $s_{i_1}\lam$ becomes $-k_{i_1}$. 

The above process is called the \emph{positive  reduction algorithm}.


\begin{Lem}[{\cite[Lem. 3.1]{do23}}]\label{find-w_lambda}
    For any integral weight $\lam$ (regular or singular),   we can get an antidominant weight $\mu$ after finite steps of  the positive reduction algorithm. We multiply all the $s_i$ appeared in these steps and get a $w_{\lambda}$. Then this $w_{\lambda}$ has the minimal length in $W$ such that $w_{\lambda}^{-1}\lam$ is antidominant.
\end{Lem}

\begin{example}
    Let  $\lam=(-1,1,-1,-1,-1,1,1,-1)$ be  an integral weight of type $E_6$. We can write $\lam=-\omega_1+2\omega_3-2\omega_4:=[-1,0,2,-2,0,0]$. The largest index of $\lam$ with a positive element is $i_1=3$. Recall that $\alpha_3=-\omega_{1}+2\omega_{3}-\omega_{4}=:[-1,0,2,-1,0,0]$. Thus $s_{3}\lam=\lam-2\alpha_3=[1,0,-2,0,0,0]$.  The largest index of $s_3\lam$ with a positive element is $i_2=1$. Recall that  $\alpha_1=2\omega_{1}-\omega_3$. Thus $s_1s_{3}\lam=s_{3}\lam-\alpha_1=[-1,0,-1,0,0,0]$, which is already antidominant.    Thus we have $w_{\lam}=s_3s_{1}$. 
\end{example}

\subsection{The case of  $\mathfrak{e}_{6(-14)}$}
In this subsection, we assume that $\mathfrak{g}$ is of type $E_6$ and $\mathfrak{g}_{\mathbb{R}}=\mathfrak{e}_{6(-14)}$.

 The root system $ \Phi $ can be realized as a subset of $\mathbb{R}^8$.  The simple roots are 
\begin{align*}
\alpha_1&=\frac{1}{2}(\ep_1-\ep_2-\ep_3-\ep_4-\ep_5-\ep_6-\ep_7+\ep_8) ,\\
\alpha_2&=\ep_1+\ep_2,\\
\alpha_k&=\ep_{k-1}-\ep_{k-2} \text{ with }3\leq k\leq 6.
\end{align*}

For type $E_6$, we have $$\Phi^+=
\{\ep_i \pm \ep_j \mid 5\geq i > j\geq 1\}\cup 
\Big\{\frac{1}{2}(\sum_{i=1}^5 (-1)^{n(i)}\ep_i -\ep_6-\ep_7+\ep_8)\mid \sum_{i=1}^5n(i) \text{ even}\Big\}$$
and $ \hs=\{ (\lam_1,\lam_2,\cdots,\lam_8)\in\mathbb{C}^8\mid \lam_6=\lam_7=-\lam_8\} $. 
Then $ \lam=(\lam_1,\lam_2,\cdots,\lam_8) $ is an integral weight if and only if $\lam_1- \lam_2-\cdots-\lam_7+\lam_8\in 2\mathbb{Z}$,  $ \lam_1-\lam_2\in\mathbb{Z}, \lam_2-\lam_3\in\mathbb{Z},\lam_3-\lam_4\in\mathbb{Z},\lam_4-\lam_5\in\mathbb{Z}$ and $2\lam_i\in \mathbb{Z} $ for $1\leq i\leq 5$.

From Enright--Howe--Wallach \cite{EHW}, we know that the highest weight module $L(\lambda)$ is an integral  Harish-Chandra module if and only if  $\lam_5>\lam_{4}>\dots>\lam_2>|\lam_1|$, $2\lam_k\in \mathbb{Z}$ for $1\leq k\leq 5$ and $\lam_i-\lam_{i+1}\in \mathbb{Z}$ for $1\leq i\leq  4$. Therefore, $L_w$ is a highest weight Harish-Chandra module of $E_6$ if and only if $-w\rho=(\lam_1,\lam_2,\cdots,\lam_n)$ with  $\lam_5>\lam_{4}>\dots>\lam_2>|\lam_1|$ and $2\lam_k\in \mathbb{Z}$ for $1\leq k\leq 5$. The number of these modules is $|\mathcal{W}|=|W(E_6)/W'_{5}|=\frac{2^7 3^{4}5}{2^4 5!}=27$.

From \cite{Bour}, the Cartan matrix of $E_6$ is $$\left.\left(\begin{array}{cccccc}2&0&-1&0&0&0\\0&2&0&-1&0&0\\-1&0&2&-1&0&0\\0&-1&-1&2&-1&0\\0&0&0&-1&2&-1\\0&0&0&0&-1&2\end{array}\right.\right).$$

First we consider the  Wallach representations. Denote $s_is_js_k:=[i,j,k]$.

\begin{Thm}\label{e6-lam}
Suppose $\frg_{\mathbb{R}}=\fre_{6(-14)}$ and $L(\lambda)$ is a highest weight Harish-Chandra module with highest weight $\lambda-\rho =-k c\zeta $. Suppose $w_{\lambda}\in W$ is the minimal length element such that $\mu=w_{\lambda}^{-1}\lambda$ is antidominant, then we have
$$w_{\lambda}=
\begin{cases}
[6,5,6,4,5,6,3,4,5,6,2,4,5,6,3,4,
           &\\5,2,4,3,1,3,4,5,6,2,4,5,3,4,2,1,3],
& \text{if~} k=1, \\
[{6},{5},{6},{4},{5}
,{6},{3},{4},{5},{6}
,{2},{4},{5},{6},&\\
 {3},{4},{5},{2},{4},{3}
,{1},{3},{4},{5},{6}
,{2},{4}],
& \text{if~} k=2.
\end{cases}
$$

\end{Thm}
\begin{proof}
    From \cite{EHW}, we have $\lambda=-k c\zeta +\rho=(0,1,2,3,4, 2k-4,2k-4,-2k+4)$.
By Proposition \ref{BCD}, we have $V(L(\lambda))=\overline{\mathcal{O}}_{k}$ for $k=0,1,2$. 

When $k=1$, we have $\lambda=(0,1,2,3,4, -2,-2,2)$. Then we can write $\lambda=-2\omega_{1}+\omega_{2}+\omega_{3}+\omega_{4}+\omega_{5}+\omega_{6}:=[-2,1,1,1,1,1]$. By Lemma \ref{find-w_lambda}, we have the following steps to find $w_{\lambda}$:
         \allowdisplaybreaks    
       \begin{align*}
      \lambda:&=[-2,1,1,1,1,1]\\
      &\xrightarrow{-\alpha_{6}}
        [-2,1,1,1,2,-1]\xrightarrow{-2\alpha_{5}}
        [-2,1,1,3,-2,1]\\&\xrightarrow{-\alpha_{6}}
        [-2,1,1,3,-1,-1]\xrightarrow{-3\alpha_{4}}
        [-2,4,4,-3,2,-1]\\&\xrightarrow{-2\alpha_{5}}
        [-2,4,4,-1,-2,1]\xrightarrow{-\alpha_{6}}
        [-2,4,4,-1,-1,-1]\\&\xrightarrow{-4\alpha_{3}}[2,4,-4,3,-1,-1] \xrightarrow{-3\alpha_{4}}
        [2,7,-1,-3,2,-1]\\&\xrightarrow{-2\alpha_{5}}
        [2,7,-1,-1,-2,1]\xrightarrow{-\alpha_{6}}
        [2,7,-1,-1,-1,-1]\\&\xrightarrow{-7\alpha_{2}}
        [2,-7,-1,6,-1,-1]\xrightarrow{-6\alpha_{4}}
        [2,-1,5,-6,5,-1]\\&\xrightarrow{-5\alpha_{5}}
        [2,-1,5,-1,-5,4]\xrightarrow{-4\alpha_{6}}
        [2,-1,5,-1,-1,-4]\\&\xrightarrow{-5\alpha_{3}}
        [7,-1,-5,4,-1,-4] \xrightarrow{-4\alpha_{4}}
        [7,3,-1,-4,3,-4]\\&\xrightarrow{-3\alpha_{5}}
        [7,3,-1,-1,-3,-1]\xrightarrow{-3\alpha_{2}}
        [7,-3,-1,2,-3,-1]\\&\xrightarrow{-2\alpha_{4}}
        [7,-1,1,-2,-1,-1] \xrightarrow{-\alpha_{3}}
        [8,-1,-1,-1,-1,-1]\\&\xrightarrow{-8\alpha_{1}}
        [-8,-1,7,-1,-1,-1]\xrightarrow{-7\alpha_{3}}
        [-1,-1,-7,6,-1,-1]\\&\xrightarrow{-6\alpha_{4}}
        [-1,5,-1,-6,5,-1]\xrightarrow{-5\alpha_{5}}
        [-1,5,-1,-1,-5,4]\\&\xrightarrow{-4\alpha_{6}}
        [-1,5,-1,-1,-1,-4]\xrightarrow{-5\alpha_{2}}
        [-1,-5,-1,4,-1,-4] \\&\xrightarrow{-4\alpha_{4}}
        [-1,-1,3,-4,3,-4]\xrightarrow{-3\alpha_{5}}
        [-1,-1,3,-1,-3,-1]\\&\xrightarrow{-3\alpha_{3}}
        [2,-1,-3,2,-3,-1] \xrightarrow{-2\alpha_{4}}
        [2,1,-1,-2,-1,-1]\\&\xrightarrow{-\alpha_{2}}
        [2,-1,-1,-1,-1,-1]\xrightarrow{-2\alpha_{1}}
        [-2,-1,1,-1,-1,-1]\\&\xrightarrow{-\alpha_{3}}
        [-1,-1,-1,0,-1,-1]\\
        &:=\mu.     
       \end{align*}
Here $x\xrightarrow{-k_j\alpha_{j}}y$ means that $s_jx=x-k_j\alpha_{j}=y$ since  $x_j=k_j$.
      
        Thus we have 
        \begin{align*}
           w_\lambda=&[6,5,6,4,5,6,3,4,5,6,2,4,5,6,3,4,\\
           &5,2,4,3,1,3,4,5,6,2,4,5,3,4,2,1,3]. 
        \end{align*}

        
When $k=2$, we have $\lambda=(0,1,2,3,4,0,0,0)$. Then we can write $\lambda=-5\omega_{1}+\omega_{2}+\omega_{3}+\omega_{4}+\omega_{5}+\omega_{6}:=[-5,1,1,1,1,1]$. By Lemma \ref{find-w_lambda}, we have the following steps to find $w_{\lambda}$:
        \allowdisplaybreaks     
       \begin{align*}
      \lambda:&=[-5,1,1,1,1,1]\\
      &\xrightarrow{-\alpha_{6}}
        [-5,1,1,1,2,-1]\xrightarrow{-2\alpha_{5}}
        [-5,1,1,3,-2,1]\\
      &\xrightarrow{-\alpha_{6}}
        [-5,1,1,3,-1,-1]\xrightarrow{-3\alpha_{4}}
        [-5,4,4,-3,2,-1]\\
      &\xrightarrow{-2\alpha_{5}}
        [-5,4,4,-1,-2,1]\xrightarrow{-\alpha_{6}}
        [-5,4,4,-1,-1,-1]\\
      &\xrightarrow{-4\alpha_{3}}[-1,4,-4,3,-1,-1]\xrightarrow{-3\alpha_{4}}
        [-1,7,-1,-3,2,-1]\\
      &\xrightarrow{-2\alpha_{5}}
        [-1,7,-1,-1,-2,1]\xrightarrow{-\alpha_{6}}
        [-1,7,-1,-1,-1,-1]\\
      &\xrightarrow{-7\alpha_{2}}
        [-1,-7,-1,6,-1,-1]\xrightarrow{-6\alpha_{4}}
        [-1,-1,5,-6,5,-1]\\
      &\xrightarrow{-5\alpha_{5}}
        [-1,-1,5,-1,-5,4]\xrightarrow{-4\alpha_{6}}
        [-1,-1,5,-1,-1,-4]\\
      &\xrightarrow{-5\alpha_{3}}
        [4,-1,-5,4,-1,-4]\xrightarrow{-4\alpha_{4}}
        [4,3,-1,-4,3,-4]\\
      &\xrightarrow{-3\alpha_{5}}
        [4,3,-1,-1,-3,-1]\xrightarrow{-3\alpha_{2}}
        [4,-3,-1,2,-3,-1]\\
      &\xrightarrow{-2\alpha_{4}}
        [4,-1,1,-2,-1,-1]\xrightarrow{-\alpha_{3}}
        [5,-1,-1,-1,-1,-1]\\
      &\xrightarrow{-5\alpha_{1}}
        [-5,-1,4,-1,-1,-1]\xrightarrow{-4\alpha_{3}}
        [-1,-1,-4,3,-1,-1]\\
      &\xrightarrow{-3\alpha_{4}}
        [-1,2,-1,-3,2,-1]\xrightarrow{-2\alpha_{5}}
        [-1,2,-1,-1,-2,1]\\
      &\xrightarrow{-\alpha_{6}}
        [-1,2,-1,-1,-1,-1]\xrightarrow{-2\alpha_{2}}
        [-1,-2,-1,1,-1,-1] \\
      &\xrightarrow{-\alpha_{4}}
        [-1,-1,0,-1,0,-1]\\
        &:=\mu.     
       \end{align*}
      
        Thus we have 
        \begin{align*}
           w_\lambda=&[{6},{5},{6},{4},{5}
,{6},{3},{4},{5},{6}
,{2},{4},{5},{6},\\
           &{3}
,{4},{5},{2},{4},{3}
,{1},{3},{4},{5},{6}
,{2},{4}]. 
        \end{align*}

\end{proof}

\begin{Cor}\label{e6-cell}
  When $\frg_{\mathbb{R}}=\fre_{6(-14)}$, let $L(\lambda)$  be an integral highest weight Harish-Chandra module. 
 Then we have $V(L(\lambda))=V(L_{w})=\overline{\mathcal{O}}_{1}$  if and only if
	  \begin{align*}
           w_\lambda\stackrel{R}{\sim}w\stackrel{R}{\sim}[6,5,6,4,5,6,3,4,5,6,2,4,5,6,3,4,
        5,2,4,3,1,3,4,5,6,2,4,5,3,4,2,1,3], 
        \end{align*}
and $V(L(\lambda))=V(L_{w})=\overline{\mathcal{O}}_{2}$  if and only if
	  \begin{align*}
           w_\lambda\stackrel{R}w{\sim}\stackrel{R}{\sim}[{6},{5},{6},{4},{5}
,{6},{3},{4},{5},{6}
,{2},{4},{5},{6},
           {3}
,{4},{5},{2},{4},{3}
,{1},{3},{4},{5},{6}
,{2},{4}]. 
        \end{align*}

\end{Cor}

    

\begin{Thm}\label{size-e6}
When $\frg_{\mathbb{R}}=\fre_{6(-14)}$, 
we use $\mathcal{C}_k$ to denote a Harish-Chandra cell consisting of  highest weight Harish-Chandra modules $L_w$ with associated variety $\overline{\mathcal{O}}_{k}$ for $0\leq k\leq 2$. Then we have
$$
\#\mathcal{C}_2=20, \#\mathcal{C}_1=6, \text{~and ~}\#\mathcal{C}_0=1.$$
\end{Thm}

\begin{proof}
    See \cite[Exa. 4.5]{BHXZ}.
\end{proof}
Note that the elements in $\mathcal{W}$ are described in \cite[Tab. 3]{Zie18}.

\subsection{The case of type $\mathfrak{e}_{7(-25)}$}
In this subsection, we assume that $\mathfrak{g}$ is of type $E_7$ and $\mathfrak{g}_{\mathbb{R}}=\mathfrak{e}_{7(-25)}$.

 The root system $ \Phi $ can be realized as a subset of $\mathbb{R}^8$.  The simple roots are 
\begin{align*}
\alpha_1&=\frac{1}{2}(\ep_1-\ep_2-\ep_3-\ep_4-\ep_5-\ep_6-\ep_7+\ep_8) ,\\
\alpha_2&=\ep_1+\ep_2,\\
\alpha_k&=\ep_{k-1}-\ep_{k-2} \text{ with }3\leq k\leq 7.
\end{align*}

For type $E_7$, we have $$\Phi^+=
\{\ep_i \pm \ep_j \mid 6\geq i > j\geq 1\}\cup 
\Big\{\frac{1}{2}(\sum_{i=1}^6 (-1)^{n(i)}\ep_i -\ep_7+\ep_8)\mid \sum_{i=1}^6 n(i) \text{ odd}\Big\}$$
and $ \hs=\{ (\lam_1,\lam_2,\cdots,\lam_8)\in\mathbb{C}^8\mid \lam_7=-\lam_8\} $. 
Then $ \lam=(\lam_1,\lam_2,\cdots,\lam_8) $ is an integral weight if and only if $\lam_1- \lam_2-\cdots-\lam_7+\lam_8\in 2\mathbb{Z}$,  $ \lam_i-\lam_{i+1}\in\mathbb{Z}$ and $2\lam_k\in \mathbb{Z} $ for $1\leq i\leq 5$ and $1\leq k\leq 6$.

From Enright--Howe--Wallach \cite{EHW}, we know that the highest weight module $L(\lambda)$ is an integral  Harish-Chandra module if and only if  $\lam_5>\lam_{4}>\dots>\lam_2>|\lam_1|$, $2\lam_k\in \mathbb{Z}$ for $1\leq k\leq 5$ and $\lam_i-\lam_{i+1}\in \mathbb{Z}$ for $1\leq i\leq  4$. Therefore, $L_w$ is a highest weight Harish-Chandra module of $E_6$ if and only if $-w\rho=(\lam_1,\lam_2,\cdots,\lam_n)$ with  $\lam_5>\lam_{4}>\dots>\lam_2>|\lam_1|$, $2\lam_k\in \mathbb{Z}$ for $1\leq k\leq 5$. The number of these modules is $|\mathcal{W}|=|W(E_7)/W(E_6)|=\frac{2^{10} 3^{4}5\cdot 7}{2^7 3^{4}5}=56$.

From \cite{Bour}, the Cartan matrix of $E_7$ is
$$\begin{pmatrix}2&0&-1&0&0&0&0\\0&2&0&-1&0&0&0\\-1&0&2&-1&0&0&0\\0&-1&-1&2&-1&0&0\\0&0&0&-1&2&-1&0\\0&0&0&0&-1&2&-1\\0&0&0&0&0&-1&2\end{pmatrix}.$$

\begin{Thm}\label{e7-lam}
Suppose $\frg_{\mathbb{R}}=\fre_{7(-25)}$ and $L(\lambda)$ is a highest weight Harish-Chandra module with highest weight $\lambda-\rho =-k c\zeta $. Suppose $w_{\lambda}\in W$ is the minimal length element such that $\mu=w_{\lambda}^{-1}\lambda$ is antidominant, then we have
$$w_{\lambda}=
\begin{cases}
[{6},{5},{6},{4},{5} ,{6},{3},{4},{5},{6}
,{7},{2},{4},{5},{6}  ,{7},{3},{4},{5},{6}
,{2},{4},{5},{3},{4}  ,{2},{1},{3},{4},{5},\\
{6},{7},{2},{4},{5}  ,{6},{3},{4},{5},{2}
,{4},{3},{1},{3},{4}  ,{5},{6},{7},{2},{4}
,{5},{6},{3},{4},{5} ,{2},{4},{3},{1}],

& \text{if~} k={1}, \\
[{6},{5},{6},{4},{5}  ,{6},{3},{4},{5},{6}
,{2},{4},{5},{6},{7}  ,{3},{4},{5},{2},{4}
 ,{3},{1},{3},{4},{5},{6},\\{7},{2},{4},{5},
{6},{3},{4},{5},{2}  ,{4},{3},{1},{3},{4}
,{5},{6},{7},{2},{4}  ,{5},{6},{3},{4},{5}
,{2}],


& \text{if~} k={2},\\
[{6},{5},{6},{4},{5}  ,{6},{3},{4},{5},{6}
,{2},{4},{5},{6},{3}  ,{4},{5},{2},{4},{3}
,{1},\\{3},{4},{5},{6}  ,{7},{2},{4},{5},{6},
{3},{4},{5},{2},{4}  ,{1},{3},{4},{5},{6}
,{7},{2}],
& \text{if~} k={3}.
\end{cases}
$$

\end{Thm}
\begin{proof}
    From \cite{EHW}, we have $\lambda=-k c\zeta +\rho=(0,1,2,3,4,5-4k,2k-\frac{17}{2},\frac{17}{2}-2k)$.
By Proposition \ref{BCD}, we have $V(L(\lambda))=\overline{\mathcal{O}}_{k}$ for $k=0,1,2,3$. 

When $k=1$, we have $\lambda=(0,1,2,3,4, 1,-\frac{13}{2},\frac{13}{2})$. Then we can write $\lambda=\omega_{1}+\omega_{2}+\omega_{3}+\omega_{4}+\omega_{5}+\omega_{6}-\omega_{7}:=[1,1,1,1,1,1,-3]$. By Lemma \ref{find-w_lambda}, we have the following steps to find $w_{\lambda}$:
   \allowdisplaybreaks               
      \begin{align*}        
        \lambda:&=[1,1,1,1,1,1,-3]\\
       & \xrightarrow{-\alpha_{6}}
        [1,	1,	1,	1,	2,	-1,	-2]\xrightarrow{-2\alpha_{5}}
        [1,	1,	1,	3,	-2,	1,	-2]\\
        &\xrightarrow{-\alpha_{6}}
        [1,	1,	1,	3,	-1,	-1,	-1]\xrightarrow{-3\alpha_{4}}
        [1,	4,	4,	-3,	2,	-1,	-1]\\
        &\xrightarrow{-2\alpha_{5}}
        [1,	4,	4,	-1,	-2,	1,	-1]\xrightarrow{-\alpha_{6}}
        [1,	4,	4,	-1,	-1,	-1,	0]\\
        &\xrightarrow{-4\alpha_{3}}
        [5,	4,	-4,	3,	-1,	-1,	0]\xrightarrow{-3\alpha_{4}}
        [5,	7,	-1,	-3,	2,	-1,	0]\\
        &\xrightarrow{-2\alpha_{5}}
        [5,	7,	-1,	-1,	-2,	1,	0]\xrightarrow{-\alpha_{6}}
        [5,	7,	-1,	-1,	-1,	-1,	1]\\
        &\xrightarrow{-\alpha_{7}}
        [5,	7,	-1,	-1,	-1,	0,	-1]\xrightarrow{-7\alpha_{2}}
        [5,	-7,	-1,	6,	-1,	0,	-1]\\
        &\xrightarrow{-6\alpha_{4}}
        [5,	-1,	5,	-6,	5,	0,	-1]\xrightarrow{-5\alpha_{5}}
        [5,	-1,	5,	-1,	-5,	5,	-1]\\
        &\xrightarrow{-5\alpha_{6}}
        [5,	-1,	5,	-1,	0,	-5,	4]\xrightarrow{-4\alpha_{7}}
        [5,	-1,	5,	-1,	0,	-1,	-4]\\
        &\xrightarrow{-5\alpha_{3}}
        [10,-1,	-5,	4,	0,	-1,	-4]\xrightarrow{-4\alpha_{4}}
        [10,3,	-1,	-4,	4,	-1,	-4]\\
        &\xrightarrow{-4\alpha_{5}}
        [10,3,	-1,	0,	-4,	3,	-4]\xrightarrow{-3\alpha_{6}}
        [10,3,	-1,	0,	-1,	-3,	-1]\\
        &\xrightarrow{-3\alpha_{2}}
        [10,-3,-1,3,	-1,	-3,	-1]\xrightarrow{-3\alpha_{4}}
        [10,0,	2,	-3,	2,	-3,	-1]\\
        &\xrightarrow{-2\alpha_{5}}
        [10,0,	2,	-1,	-2,	-1,	-1]\xrightarrow{-2\alpha_{3}}
        [12,0,	-2,	1,	-2,	-1,	-1]\\
        &\xrightarrow{-\alpha_{4}}
        [12,1,	-1,	-1,	-1,	-1,	-1]\xrightarrow{-\alpha_{2}}
        [12,-1,	-1,	0,	-1,	-1,	-1]\\
        &\xrightarrow{-12\alpha_{1}}
        [-12,-1, 11,0,	-1,	-1,	-1]\xrightarrow{-11\alpha_{3}}
        [-1,-1,	-11,11,	-1,	-1,	-1]\\
        &\xrightarrow{-11\alpha_{4}}
        [-1,10,0,-11,10,-1,	-1]\xrightarrow{-10\alpha_{5}}
        [-1,10,0,-1,-10,9,	-1]\\
        &\xrightarrow{-9\alpha_{6}}
        [-1,10,0,-1,-1,-9,8]\xrightarrow{-8\alpha_{7}}
        [-1,10,0,-1,-1,-1,-8]\\
        &\xrightarrow{-10\alpha_{2}}
        [-1,-10,0,9,-1,	-1,	-8]\xrightarrow{-9\alpha_{4}}
        [-1,-1,	9,-9,8,-1,-8]\\
        &\xrightarrow{-8\alpha_{5}}
        [-1,-1,9,-1,-8,7,-8]\xrightarrow{-7\alpha_{6}}
        [-1,-1,9,-1,-1,-7,-1]\\
        &\xrightarrow{-9\alpha_{3}}
        [8,-1,-9,8,-1,-7,-1]\xrightarrow{-8\alpha_{4}}
        [8,7,-1,-8,7,-7,-1]\\
        &\xrightarrow{-7\alpha_{5}}
        [8,7,-1,-1,-7,0,-1]\xrightarrow{-7\alpha_{2}}
        [8,	-7,	-1,	6,	-7,	0,	-1]\\
        &\xrightarrow{-6\alpha_{4}}
        [8,	-1,	5,	-6,	-1,	0,	-1]\xrightarrow{-5\alpha_{3}}
        [13,	-1,	-5,	-1,	-1,	0,	-1]\\
        &\xrightarrow{-13\alpha_{1}}
        [-13,	-1,	8,	-1,	-1,	0,	-1]\xrightarrow{-8\alpha_{3}}
        [-5,	-1,	-8,	7,	-1,	0,	-1]\\
        &\xrightarrow{-7\alpha_{4}}
        [-5,	6,	-1,	-7,	6,	0,	-1]\xrightarrow{-6\alpha_{5}}
        [-5,	6,	-1,	-1,	-6,	6,	-1]\\
        &\xrightarrow{-6\alpha_{6}}
        [-5,	6,	-1,	-1,	0,	-6,	5]\xrightarrow{-5\alpha_{7}}
        [-5,	6,	-1,	-1,	0,	-1,	-5]\\
        &\xrightarrow{-6\alpha_{2}}
        [-5,	-6,	-1,	5,	0,	-1,	-5]\xrightarrow{-5\alpha_{4}}
        [-5,	-1,	4,	-5,	5,	-1,	-5]\\
        &\xrightarrow{-5\alpha_{5}}
        [-5,	-1,	4,	0,	-5,	4,	-5]\xrightarrow{-4\alpha_{6}}
        [-5,	-1,	4,	0,	-1,	-4,	-1]\\
        &\xrightarrow{-4\alpha_{3}}
        [-1,	-1,	-4,	4,	-1,	-4,	-1]\xrightarrow{-4\alpha_{4}}
        [-1,	3,	0,	-4,	3,	-4,	-1]\\
        &\xrightarrow{-3\alpha_{5}}
        [-1,	3,	0,	-1,	-3,	-1,	-1]\xrightarrow{-2\alpha_{2}}
        [-1	,-3,	0,	2,	-3,	-1,	-1]\\
        &\xrightarrow{-2\alpha_{4}}
        [-1,	-1,	2,	-2,	-1,	-1,	-1]\xrightarrow{-2\alpha_{3}}
        [1,	-1,	-2,	0,	-1,	-1,	-1]\\
        &\xrightarrow{-\alpha_{1}}
        [-1,	-1,	-1,	0,	-1,-1,	-1]\\
        &:=\mu.
\end{align*}

        Thus we have 
        \begin{align*}
           w_\lambda=&[{6},{5},{6},{4},{5} ,{6},{3},{4},{5},{6}
,{7},{2},{4},{5},{6}  ,{7},{3},{4},{5},{6}
,{2},{4},{5},{3},{4}  ,{2},{1},{3},{4},{5},\\&
{6},{7},{2},{4},{5}  ,{6},{3},{4},{5},{2}
,{4},{3},{1},{3},{4}  ,{5},{6},{7},{2},{4}
,{5},{6},{3},{4},{5} ,{2},{4},{3},{1}]. 
        \end{align*}

        
When $k=2$, we have $\lambda=(0,1,2,3,4,-3,-\frac{9}{2},\frac{9}{2})$. Then we can write $\lambda=\omega_{1}+\omega_{2}+\omega_{3}+\omega_{4}+\omega_{5}+\omega_{6}-7\omega_{7}:=[1,1,1,1,1,1,-7]$. By Lemma \ref{find-w_lambda}, we have the following steps to find $w_{\lambda}$:
 \allowdisplaybreaks  
  \begin{align*}
   \lambda:&=[1,1,1,1,1,1,-7]\\
   &\xrightarrow{-\alpha_{6}}
        [1,	1,	1,	1,	2,	-1,	-6]\xrightarrow{-2\alpha_{5}}
        [1,	1,	1,	3,	-2,	1,	-6]\\
        &\xrightarrow{-\alpha_{6}}
      [1,	1,	1,	3,	-1,	-1,	-5]\xrightarrow{-3\alpha_{4}}
        [1,	4,	4,	-3,	2,	-1,	-5]\\
        &\xrightarrow{-2\alpha_{5}}
        [1,	4,	4,	-1,	-2,	1,	-5]\xrightarrow{-\alpha_{6}}
        [1,	4,	4,	-1,	-1,	-1,	-4]\\
        &\xrightarrow{-4\alpha_{3}}
        [5,	4,	-4,	3,	-1,	-1,	-4]\xrightarrow{-3\alpha_{4}}
        [5,	7,	-1,	-3,	2,	-1,	-4]\\
        &\xrightarrow{-2\alpha_{5}}
        [5,	7,	-1,	-1,	-2,	1,	-4]\xrightarrow{-\alpha_{6}}
        [5,	7,	-1,	-1,	-1,	-1,	-3]\\
        &\xrightarrow{-7\alpha_{2}}
        [5,	-7,	-1,	6,	-1,	-1,	-3]\xrightarrow{-6\alpha_{4}}
       [5,	-1,	5,	-6,	5,	-1,	-3]\\
       &\xrightarrow{-5\alpha_{5}}
        [5,	-1,	5,	-1,	-5,	4,	-3]\xrightarrow{-4\alpha_{6}}
        [5,	-1,	5,	-1,	-1,	-4,	1]\\
        &\xrightarrow{-\alpha_{7}}
        [5,	-1,	5,	-1,	-1,	-3,	-1]\xrightarrow{-5\alpha_{3}}
        [10,	-1,	-5,	4,	-1,	-3,	-1]\\
        &\xrightarrow{-4\alpha_{4}}
        [10,3,	-1,	-4,	3,	-3,	-1]\xrightarrow{-3\alpha_{5}}
        [10,3,	-1,	-1,	-3,	0,	-1]\\
        &\xrightarrow{-3\alpha_{2}}
        [10,-3,	-1,	2,	-3,	0,	-1]\xrightarrow{-2\alpha_{4}}
        [10,-1,1,	-2,	-1,	0,	-1]\\
        &\xrightarrow{-\alpha_{3}}
        [11,-1,-1,-1,	-1,0,	-1]\xrightarrow{-11\alpha_{1}}
        [-11,-1,10,-1,-1,0,-1]\\
        &\xrightarrow{-10\alpha_{3}}
        [-1,-1,-10,9,-1,0,-1]\xrightarrow{-9\alpha_{4}}
       [-1,8,-1,-9,8,0,-1]\\
       &\xrightarrow{-8\alpha_{5}}
        [-1,8,-1,-1,-8,8,-1]\xrightarrow{-8\alpha_{6}}
        [-1,8,-1,-1,0,-8,7]\\
        &\xrightarrow{-7\alpha_{7}}
        [-1,8,-1,-1,0,-1,-7]\xrightarrow{-8\alpha_{2}}
        [-1,-8,-1,7,0,-1,-7]\\
        &\xrightarrow{-7\alpha_{4}}
        [-1,-1,6,-7,7,-1,-7]\xrightarrow{-7\alpha_{5}}
        [-1,-1,6,0,-7,6,-7]\\
        &\xrightarrow{-6\alpha_{6}}
        [-1,-1,6,0,-1,-6,-1]\xrightarrow{-6\alpha_{3}}
        [5,-1,-6,6,-1,-6,-1]\\
        &\xrightarrow{-6\alpha_{4}}
        [5,5,0,-6,5,-6,-1]\xrightarrow{-5\alpha_{5}}
        [5,5,0,-1,-5,-1,-1]\\
        &\xrightarrow{-5\alpha_{2}}
        [5,-5,0,4,-5,-1,-1]\xrightarrow{-4\alpha_{4}}
       [5,-1,4,-4,-1,-1,-1]\\
       &\xrightarrow{-4\alpha_{3}}
       [9,-1,-4,0,-1,-1,-1]\xrightarrow{-9\alpha_{1}}
        [-9,-1,5,0,-1,-1,-1]\\
        &\xrightarrow{-5\alpha_{3}}
        [-4,	-1,	-5,	5,	-1,	-1,	-1]\xrightarrow{-5\alpha_{4}}
        [-4,	4,	0,	-5,	4,	-1,	-1]\\
        &\xrightarrow{-4\alpha_{5}}
         [-4	,4,	0,	-1,	-4,	3,	-1]\xrightarrow{-3\alpha_{6}}  
        [-4,	4,	0,	-1,	-1,	-3,	2]\\
        &\xrightarrow{-2\alpha_{7}}
        [-4,	4,	0,	-1,	-1,	-1,	-2]\xrightarrow{-4\alpha_{2}}
        [-4,	-4,	0,	3,	-1,	-1,	-2]\\
        &\xrightarrow{-3\alpha_{4}}
        [-4	,-1,	3,	-3	,2,	-1,	-2]\xrightarrow{-2\alpha_{5}}
        [-4,	-1,	3,	-1,	-2,	1,	-2]\\
        &\xrightarrow{-\alpha_{6}}
        [-4	,-1,	3,	-1,	-1,	-1,	-1]\xrightarrow{-3\alpha_{3}}
        [-1,	-1,	-3,	2,	-1,	-1,	-1]\\
        &\xrightarrow{-2\alpha_{4}}
        [-1	,1,	-1,	-2,	1,	-1,	-1]\xrightarrow{-\alpha_{5}}
        [-1,	1,	-1,	-1,	-1,	0,	-1]\\
        &\xrightarrow{-\alpha_{2}}
        [-1,	-1,	-1,	0,	-1,	0,	-1]\\
        &:=\mu.  
  \end{align*}

  Thus we have 
        \begin{align*}
           w_\lambda=&[{6},{5},{6},{4},{5},{6},{3},{4},{5},{6},{2},{4},{5},{6},{7},{3},{4},{5},{2},{4},{3},{1},{3},{4},{5},{6},\\&{7},{2},{4},{5},
{6},{3},{4},{5},{2},{4},{3},{1},{3},{4},{5},{6},{7},{2},{4},{5},{6},{3},{4},{5},{2}]. 
        \end{align*}

When $k=3$, we have $\lambda=(0,1,2,3,4,-7,-\frac{5}{2},\frac{5}{2})$. Then we can write $\lambda=\omega_{1}+\omega_{2}+\omega_{3}+\omega_{4}+\omega_{5}+\omega_{6}-11\omega_{7}:=[1,1,1,1,1,1,-11]$. By Lemma \ref{find-w_lambda}, we have the following steps to find $w_{\lambda}$:
   \allowdisplaybreaks            
\begin{align*}
    \lambda:&=[1,1,1,1,1,1,-11]\\
    &\xrightarrow{-\alpha_{6}}
        [1,	1,	1,	1,	2,	-1,	-10]\xrightarrow{-2\alpha_{5}}
        [1,	1,	1,	3,	-2,	1,	-10]\\
        &\xrightarrow{-\alpha_{6}}
        [1,	1,	1,	3,	-1,	-1,	-9]\xrightarrow{-3\alpha_{4}}
        [1,	4,	4,	-3,	2,	-1,	-9]\\
        &\xrightarrow{-2\alpha_{5}}
        [1,	4,	4,	-1,	-2,	1,	-9]\xrightarrow{-\alpha_{6}}
        [1,	4,	4,	-1,	-1,	-1,	-8]\\
        &\xrightarrow{-4\alpha_{3}}
        [5,	4,	-4,	3,	-1,	-1,	-8]\xrightarrow{-3\alpha_{4}}
        [5,	7,	-1,	-3,	2,	-1,	-8]\\
        &\xrightarrow{-2\alpha_{5}}
        [5,	7,	-1,	-1,	-2,	1,	-8]\xrightarrow{-\alpha_{6}}
        [5,	7,	-1,	-1,	-1,	-1,	-7]\\
        &\xrightarrow{-7\alpha_{2}}
        [5,	-7,	-1,	6,	-1,	-1,	-7]\xrightarrow{-6\alpha_{4}}
        [5,	-1,	5,	-6,	5,	-1,	-7]\\
        &\xrightarrow{-5\alpha_{5}}
        [5,	-1,	5,	-1,	-5,	4,	-7]\xrightarrow{-4\alpha_{6}}
        [5,	-1,	5,	-1,	-1,	-4,	-3]\\
        &\xrightarrow{-5\alpha_{3}}
        [10,	-1,	-5,	4,	-1,	-4,	-3]\xrightarrow{-4\alpha_{4}}
        [10,	3,	-1,	-4,	3,	-4,	-3]\\
        &\xrightarrow{-3\alpha_{5}}
        [10,	3,	-1,	-1,	-3,	-1,	-3]\xrightarrow{-3\alpha_{2}}
        [10,	-3,	-1,	2,	-3,	-1,	-3]\\
        &\xrightarrow{-2\alpha_{4}}
        [10,	-1,	1,	-2,	-1,	-1,	-3]\xrightarrow{-\alpha_{3}}
        [11,	-1,	-1,	-1,	-1,	-1,	-3]\\
        &\xrightarrow{-11\alpha_{1}}
        [-11,	-1,	10,	-1,	-1,	-1,	-3]\xrightarrow{-10\alpha_{3}}
        [-1,-1,	-10,	9,	-1,	-1,	-3]\\
        &\xrightarrow{-9\alpha_{4}}
        [-1,	8,	-1,	-9,	8,	-1,	-3]\xrightarrow{-8\alpha_{5}}
        [-1,	8,	-1,	-1,	-8,	7,	-3]\\
        &\xrightarrow{-7\alpha_{6}}
        [-1,	8,	-1,	-1,	-1,	-7,	4]\xrightarrow{-4\alpha_{7}}
        [-1,	8,	-1,	-1,	-1,	-3,	-4]\\
        &\xrightarrow{-8\alpha_{2}}
        [-1,	-8,	-1,	7,	-1,	-3,	-4]\xrightarrow{-7\alpha_{4}}
        [-1,	-1,	6,	-7,	6,	-3,	-4]\\
        &\xrightarrow{-6\alpha_{5}}
        [-1,	-1,	6,	-1,	-6,	3,	-4]\xrightarrow{-3\alpha_{6}}
        [-1,	-1,	6,	-1,	-3,	-3,	-1]\\
        &\xrightarrow{-6\alpha_{3}}
        [5,	-1,	-6,	5,	-3,	-3,	-1]\xrightarrow{-5\alpha_{4}}
        [5,	4,	-1,	-5,	2,	-3,	-1]\\
        &\xrightarrow{-2\alpha_{5}}
        [5,	4,	-1,	-3,	-2,	-1,	-1]\xrightarrow{-4\alpha_{2}}
        [5,	-4,	-1,	1,	-2,	-1,	-1]\\
        &\xrightarrow{-\alpha_{4}}
        [5,	-3,	0,	-1,	-1,	-1,	-1]\xrightarrow{-5\alpha_{1}}
        [-5,	-3,	5,	-1,	-1,	-1,	-1]\\
        &\xrightarrow{-5\alpha_{3}}
        [0,	-3,	-5,	4,	-1,	-1,	-1]\xrightarrow{-4\alpha_{4}}
        [0,	1,	-1,	-4,	3,	-1,	-1]\\
        &\xrightarrow{-3\alpha_{5}}
        [0,	1,	-1,	-1,	-3,	2,	-1]\xrightarrow{-2\alpha_{6}}
        [0,	1,	-1,	-1,	-1,	-2,	1]\\
        &\xrightarrow{-\alpha_{7}}
        [0,	1,	-1,	-1,	-1,	-1,	-1]\xrightarrow{-\alpha_{2}}
        [0,	-1,	-1,	0,	-1,	-1,	-1]\\
        &:=\mu.  
  \end{align*}
    
 Thus we have 
        \begin{align*}
           w_\lambda=&[{6},{5},{6},{4},{5},{6},{3},{4},{5},{6},{2},{4},{5},{6},{3},{4},{5},{2},{4},{3},{1},\\&{3},{4},{5},{6},{7},{2},{4},{5},{6},
{3},{4},{5},{2},{4},{1},{3},{4},{5},{6},{7},{2}]. 
        \end{align*}

\end{proof}

\begin{Cor}\label{e7-cell}
 When $\frg_{\mathbb{R}}=\fre_{7(-25)}$,  let $L(\lambda)$  be an integral highest weight Harish-Chandra module. 
 Then we have $V(L(\lambda))=V(L_{w})=\overline{\mathcal{O}}_{1}$  if and only if
	   \begin{align*}
           w_\lambda\stackrel{R}{\sim}w\stackrel{R}{\sim}&[{6},{5},{6},{4},{5},{6},{3},{4},{5},{6},{7},{2},{4},{5},{6},{7},{3},{4},{5},{6},{2},{4},{5},{3},{4},{2},{1},{3},{4},{5},\\
&{6},{7},{2},{4},{5},{6},{3},{4},{5},{2},{4},{3},{1},{3},{4},{5},{6},{7},{2},{4},{5},{6},{3},{4},{5},{2},{4},{3},{1}], 
        \end{align*}
and $V(L(\lambda))=V(L_{w})=\overline{\mathcal{O}}_{2}$  if and only if
	 \begin{align*}
           w_\lambda\stackrel{R}{\sim}w\stackrel{R}{\sim}&[{6},{5},{6},{4},{5},{6},{3},{4},{5},{6},{2},{4},{5},{6},{7},{3},{4},{5},{2},{4},{3},{1},{3},{4},{5},{6},\\
&{7},{2},{4},{5},{6},{3},{4},{5},{2},{4},{3},{1},{3},{4},{5},{6},{7},{2},{4},{5},{6},{3},{4},{5},{2}],
\end{align*}
and $V(L(\lambda))=V(L_{w})=\overline{\mathcal{O}}_{3}$ if and only if
	  \begin{align*}
           w_\lambda\stackrel{R}{\sim}w\stackrel{R}{\sim}&[{6},{5},{6},{4},{5},{6},{3},{4},{5},{6},{2},{4},{5},{6},{3},{4},{5},{2},{4},{3},{1},\\&{3},{4},{5},{6},{7},{2},{4},{5},{6},{3},{4},{5},{2},{4},{1},{3},{4},{5},{6},{7},{2}].
        \end{align*}

\end{Cor}

\begin{Thm}\label{size-e7}
When $\frg_{\mathbb{R}}=\fre_{7(-25)}$,
we use $\mathcal{C}_k$ to denote a Harish-Chandra cell consisting of  highest weight Harish-Chandra modules $L_w$ with associated variety $\overline{\mathcal{O}}_{k}$ for $0\leq k\leq 3$. Then we have
$$
\#\mathcal{C}_3=21, \#\mathcal{C}_2=27,\#\mathcal{C}_1=7 \text{~and ~}\#\mathcal{C}_0=1.$$
\end{Thm}

\begin{proof}
    See Example \ref{e7-cell-eq}.
\end{proof}
Note that the elements in $\mathcal{W}$ are described in \cite[Tab. 4]{Zie18}.

\section{The cases of type \texorpdfstring{$B_n$}{} and type \texorpdfstring{$C_n$}{}}\label{bc-thm}
In this section, we consider highest weight Harish-Chandra modules of types $B_n$ and $C_n$.
We use $W_n$ to denote their Weyl groups.

\subsection{The case of $\mathfrak{so}(2,2n-1)$}
In this subsection, we assume $\mathfrak{g}=\mathfrak{so}(2n+1,\mathbb{C})$ and $\mathfrak{g}_{\mathbb{R}}=\mathfrak{so}(2,2n-1)$. 
 From Enright--Howe--Wallach \cite{EHW}, we know that the highest weight module $L(\lambda)$ is an integral  Harish-Chandra module if and only if $\lambda_1-\lam_2\in \mathbb{Z}$, $\lam_i-\lam_{i+1}\in\mathbb{Z}_{>0}$ for $2\leq i\leq n-1$ and $2\lambda_k\in \mathbb{Z}_{>0}$ for $2\leq k\leq n$. Therefore, $L_w$ is a highest weight Harish-Chandra $SO(2,2n-1)$-module if and only if $-w\rho=(\lam_1,\lam_2,\cdots,\lam_n)$ with $\lam_2>\lam_3>\cdots>\lam_n>0$. The number of these modules is $|\mathcal{W}|=|W_n/W_{n-1}|=2n$.


First we consider the Wallach representations.

\begin{Thm}\label{b-2-cell}
Suppose $\frg_{\mathbb{R}}=\frso(2,2n-1)$ and $L(\lambda)$ is a highest weight Harish-Chandra module with highest weight $\lambda-\rho =-k c\zeta $. Suppose $w_{\lambda}\in W$ is the minimal length element such that $\mu=w_{\lambda}^{-1}\lambda$ is antidominant, then we have
\begin{enumerate}
    \item if $k=2$, then $$w_{\lambda}=(-1,-2, \dots ,-(n-2),n,-(n-1)) \text{~ and~} V(L(\lambda))=\overline{\mathcal{O}}_{2}.$$
     \item if $k=0$, then $$w_{\lambda}=-Id \text{~ and~} V(L(\lambda))=\overline{\mathcal{O}}_{0}.$$
\end{enumerate}

\end{Thm}
\begin{proof}
    From \cite{EHW}, we have $\lambda=-k c\zeta +\rho=(-k(n-\frac{3}{2})+(n-\frac{1}{2}),n-\frac{3}{2},\dots, \frac{1}{2})$.
  When $k=2$, $\lambda=(-(n-\frac{5}{2}),n-\frac{3}{2},n-\frac{5}{2},\dots,\frac{3}{2},\frac{1}{2}) $. By Proposition \ref{BCD}, we have $V(L(\lambda))=\overline{\mathcal{O}}_{2}.$ By Proposition \ref{anti},  we have $\mu=(-(n-\frac{5}{2}),-(n-\frac{5}{2}), -(n-\frac{3}{2}),\dots, -\frac{3}{2},-\frac{1}{2})$. Thus we have:
         $$\begin{tikzpicture}
     \draw(-1,-4) node{$ \mu^-: $};
      \draw (-1,0) node{$ \lambda^-: $};
    \node (num1) at (0.5,0) {$-(n-\frac{5}{2})$};
    \node (num2) at (2,0) {$n-\frac{3}{2}$};
    \node (num3) at (3,0) {$\dots$};
    \node (num4) at (4,0) {$\frac{3}{2}$};
    \node (num5) at (5,0) {$\frac{1}{2}$};
    \node (num6) at (6,0) {$-\frac{1}{2}$};
    \node (num7) at (7,0) {$-\frac{3}{2}$};
    \node (num8) at (8,0) {$\dots$};
    \node (num9) at (9,0) {$-(n-\frac{3}{2})$};
    \node (num10) at (10.5,0) {$n-\frac{5}{2}$};
    
    \node (num1') at (0.5,-4) {$-(n-\frac{3}{2})$};
    \node (num2') at (2,-4) {$-(n-\frac{5}{2})$};
    \node (num3') at (3,-4) {$\dots$};
    \node (num4') at (4,-4) {$-\frac{3}{2}$};
    \node (num5') at (5,-4) {$-\frac{1}{2}$};
    \node (num6') at (6,-4) {$\frac{1}{2}$};
    \node (num7') at (7,-4) {$\frac{3}{2}$};
    \node (num8') at (8,-4) {$\dots$};
    \node (num9') at (9,-4) {$n-\frac{5}{2}$};
    \node (num10') at (10.5,-4) {$n-\frac{3}{2}$};
    
    \draw(num1') -- (num9);
    \draw(num2') -- (num1);
    \draw(num4') -- (num7);
    \draw (num5') -- (num6);
   \draw(num6') -- (num5);
    \draw(num7') -- (num4);    
    \draw(num9') -- (num10);    
    \draw(num10') -- (num2);

    \draw[dashed](5.5,0)--(5.5,-4);
\end{tikzpicture}$$  
By Lemma \ref{lambda-w}, this diagram represents an element $$w_\lambda =(-1,-2, \dots ,-(n-2),n,-(n-1)).$$ By using the RS algorithm, we have
\[P(^{-}w_{\lambda})=\tiny{\begin{tikzpicture}[scale=0.8,baseline=-50pt]
			\hobox{0}{0}{-n}
			\hobox{0}{1}{-n+2}
			\hobox{0}{2}{\vdots}
                \hobox{0}{3}{n-1}
			\hobox{1}{0}{-n+1}
			\hobox{2}{0}{n}
   \end{tikzpicture}}.\]
 
When $k=0$, we have $\lambda=\rho$, thus $w_{\lambda}=-Id$.

    \end{proof}

In general, we have the following.
\begin{Thm}\label{b-cell-3}
Suppose $\frg_{\mathbb{R}}=\frso(2,2n-1)$ and $L(\lambda)$ is a highest weight Harish-Chandra module with highest weight $\lambda-\rho$.
Denote $\lambda=(\lam_{1},...\lam_{n})$.  Suppose $w_{\lambda}\in W$ is the minimal length element such that $\mu=w_{\lambda}^{-1}\lambda$ is antidominant, then we have 
\begin{enumerate}
    \item if $\lam_{i}\geq \lam_{1}>\lam_{i+1} (2\leq i<n)$, then $$w_{\lambda}=(-1, \dots, -(n-i),-n,-(n-i+1),\dots, -(n-1)) \text{~ and~} V(L(\lambda))=\overline{\mathcal{O}}_{2}.$$
     \item if $\lam_{n}\geq \lam_{1}>0$, then $$w_{\lambda}=(-n,-1,\dots,-(n-1 )) \text{~ and~} V(L(\lambda))=\overline{\mathcal{O}}_{2}.$$

     \item if $0\geq \lam_{1}>-\lam_{n}$, then $$w_{\lambda}=(n,-1,\dots, -(n-2),-(n-1 )) \text{~ and~} V(L(\lambda))=\overline{\mathcal{O}}_{2}.$$
     \item if   $-\lam_{i+1}\geq \lam_1>-\lam_i$ for some $2\leq i\leq n-1$, then $$w_{\lambda}=(-1,\dots, -(n-i),n,-(n-i+1),\dots, -(n-1)) \text{~ and~} V(L(\lambda))=\overline{\mathcal{O}}_{2}.$$
     \item if  $-\lam_{2}\geq \lam_1$, then 
     $$w_\lambda=(-1,\dots,-(n-1),n) \text{~ and~} V(L(\lambda))=\overline{\mathcal{O}}_{2}.$$
\end{enumerate}


\end{Thm}

\begin{proof}

If $\lam_{i}\geq \lam_{1}>\lam_{i+1} (2\leq i<n)$, we will have:
$$\tiny{\begin{tikzpicture}
     \draw(0,-4) node{$ \mu^-: $};
      \draw (0,0) node{$ \lambda^-: $};
    \node (num1) at (0.8,0) {$\lam_{1}$};
    \node (num2) at (1.6,0) {$\lam_{2}$};
    \node (num3) at (2.4,0) {$\dots$};
    \node (num4) at (3.2,0) {$\lam_{i}$};
    \node (num5) at (4,0) {$\lam_{i+1}$};
    \node (num6) at (5,0) {$\dots$};
    \node (num7) at (5.8,0) {$\lam_{n}$};
    \node (num8) at (6.8,0) {$-\lam_{n}$};
    \node (num9) at (7.6,0) {$\dots$};
    \node (num10) at (8.4,0) {$-\lam_{i+1}$};
    \node (num11) at (9.2,0) {$-\lam_{i}$};
      \node (num12) at (10,0) {$\dots$};
     \node (num13) at (10.8,0) {$-\lam_{2}$};
      \node (num14) at (11.6,0) {$-\lam_{1}$};

    \node (num1') at (0.8,-4) {$-\lam_{2}$};
    \node (num2') at (1.6,-4) {$\dots$};
    \node (num3') at (2.4,-4) {$-\lam_{i}$};
    \node (num4') at (3.2,-4) {$-\lam_{1}$};
    \node (num5') at (4,-4) {$-\lam_{i+1}$};
    \node (num6') at (5,-4) {$\dots$};
    \node (num7') at (5.8,-4) {$-\lam_{n}$};
    \node (num8') at (6.8,-4) {$\lam_{n}$};
    \node (num9') at (7.6,-4) {$\dots$};
    \node (num10') at (8.4,-4) {$\lam_{i+1}$};
    \node (num11') at (9.2,-4) {$\lam_{1}$};
    \node (num12') at (10,-4) {$\lam_{i}$};
    \node (num13') at (10.8,-4) {$\dots$};
    \node (num14') at (11.6,-4) {$\lam_{2}$};

    \draw(num1') -- (num13);
    \draw(num3') -- (num11);
    \draw(num4') -- (num14);
    \draw (num5') -- (num10);
   \draw(num7') -- (num8);
    \draw(num8') -- (num7);    
    \draw(num10') -- (num5);    
    \draw(num11') -- (num1);
    \draw(num12') -- (num4);
     \draw(num14') -- (num2);

    \draw[dashed](6.3,0.3)--(6.3,-4.3);
\end{tikzpicture}}$$

By Lemma \ref{lambda-w}, we can find that $$w_\lambda=(-1,\dots,-(n-i),-n,-(n-i+1),\dots,-(n-1)).$$ By using the RS algorithm, we have
\[P(^{-}w_{\lambda})=\tiny{\begin{tikzpicture}[scale=0.8,baseline=-50pt]
			\hobox{0}{0}{-n}
			\hobox{0}{1}{-n+2}
			\hobox{0}{2}{\vdots}
                \hobox{0}{3}{n-1}
			\hobox{1}{0}{-n+1}
			\hobox{1}{1}{n}
   \end{tikzpicture}} .\]
            
If $\lam_{n}\geq \lam_{1}>0$, we will have 
$$\begin{tikzpicture}
     \draw(0,-3) node{$ \mu^-: $};
      \draw (0,0) node{$ \lambda^-: $};
    \node (num1) at (1,0) {$\lam_{1}$};
    \node (num2) at (2,0) {$\lam_{2}$};
    \node (num3) at (3,0) {$\dots$};
    \node (num4) at (4,0) {$\lam_{n}$};
    \node (num5) at (5,0) {$-\lam_{n}$};
    \node (num6) at (6,0) {$\dots$};
    \node (num7) at (7,0) {$-\lam_{2}$};
    \node (num8) at (8,0) {$-\lam_{1}$};

    \node (num1') at (1,-3) {$-\lam_{2}$};
    \node (num2') at (2,-3) {$\dots$};
    \node (num3') at (3,-3) {$-\lam_{n}$};
    \node (num4') at (4,-3) {$-\lam_{1}$};
    \node (num5') at (5,-3) {$\lam_{1}$};
    \node (num6') at (6,-3) {$\lam_{n}$};
    \node (num7') at (7,-3) {$\dots$};
    \node (num8') at (8,-3) {$\lam_{2}$};

     \draw(num1') -- (num7);
    \draw(num3') -- (num5);
    \draw(num4') -- (num8);
    \draw (num5') -- (num1);
   \draw(num6') -- (num4);
    \draw(num8') -- (num2);    
    \draw[dashed](4.5,0.3)--(4.5,-3.3);
\end{tikzpicture}$$ 
By Lemma \ref{lambda-w}, we can find that $$w_{\lambda}=(-n,-1,\dots,-(n-1 )).$$  In this case, $P(^{-}w_{\lambda})$ is the same as the previous one.

If $0\geq \lam_{1}>-\lam_{n}$, we will have:
$$\begin{tikzpicture}
     \draw(0,-3) node{$ \mu^-: $};
      \draw (0,0) node{$ \lambda^-: $};
    \node (num1) at (1,0) {$\lam_{1}$};
    \node (num2) at (2,0) {$\lam_{2}$};
    \node (num3) at (3,0) {$\dots$};
    \node (num4) at (4,0) {$\lam_{n}$};
    \node (num5) at (5,0) {$-\lam_{n}$};
    \node (num6) at (6,0) {$\dots$};
    \node (num7) at (7,0) {$-\lam_{2}$};
    \node (num8) at (8,0) {$-\lam_{1}$};

            \node (num1') at (1,-3) {$-\lam_{2}$};
    \node (num2') at (2,-3) {$\dots$};
    \node (num3') at (3,-3) {$-\lam_{n}$};
    \node (num4') at (4,-3) {$\lam_{1}$};
    \node (num5') at (5,-3) {$-\lam_{1}$};
    \node (num6') at (6,-3) {$\lam_{n}$};
    \node (num7') at (7,-3) {$\dots$};
    \node (num8') at (8,-3) {$\lam_{2}$};

        \draw(num1') -- (num7);
    \draw(num3') -- (num5);
    \draw(num4') -- (num1);
    \draw (num5') -- (num8);
   \draw(num6') -- (num4);
    \draw(num8') -- (num2);    
    \draw[dashed](4.5,0.3)--(4.5,-3.3);
\end{tikzpicture}$$ 
By Lemma \ref{lambda-w}, we can find that $$w_\lambda=(n,-1,\dots,-(n-2),-(n-1)).$$  In this case, $P(^{-}w_{\lambda})$ is the same as the previous one.

If  $-\lam_{i+1}\geq \lam_1>-\lam_i$ for some $2\leq i\leq n$,  we will have:
$$\tiny{\begin{tikzpicture}
     \draw(0,-4) node{$ \mu^-: $};
      \draw (0,0) node{$ \lambda^-: $};
    \node (num1) at (0.7,0) {$\lam_{1}$};
    \node (num2) at (1.5,0) {$\lam_{2}$};
    \node (num3) at (2.3,0) {$\dots$};
    \node (num4) at (3.2,0) {$\lam_{i}$};
    \node (num5) at (4,0) {$\lam_{i+1}$};
    \node (num6) at (5,0) {$\dots$};
    \node (num7) at (5.8,0) {$\lam_{n}$};
    \node (num8) at (6.8,0) {$-\lam_{n}$};
    \node (num9) at (7.6,0) {$\dots$};
    \node (num10) at (8.4,0) {$-\lam_{i+1}$};
    \node (num11) at (9.3,0) {$-\lam_{i}$};
      \node (num12) at (10,0) {$\dots$};
     \node (num13) at (10.8,0) {$-\lam_{2}$};
      \node (num14) at (11.6,0) {$-\lam_{1}$};

        \node (num1') at (0.7,-4) {$-\lam_{2}$};
    \node (num2') at (1.5,-4) {$\dots$};
    \node (num3') at (2.3,-4) {$-\lam_{i}$};
    \node (num4') at (3.2,-4) {$\lam_{1}$};
    \node (num5') at (4,-4) {$-\lam_{i+1}$};
    \node (num6') at (4.9,-4) {$\dots$};
    \node (num7') at (5.8,-4) {$-\lam_{n}$};
    \node (num8') at (6.8,-4) {$\lam_{n}$};
    \node (num9') at (7.6,-4) {$\dots$};
    \node (num10') at (8.4,-4) {$\lam_{i+1}$};
    \node (num11') at (9.2,-4) {$-\lam_{1}$};
    \node (num12') at (10.1,-4) {$\lam_{i}$};
    \node (num13') at (10.8,-4) {$\dots$};
    \node (num14') at (11.6,-4) {$\lam_{2}$};
    
   \draw(num1') -- (num13);
    \draw(num3') -- (num11);
    \draw(num4') -- (num1);
    \draw (num5') -- (num10);
   \draw(num7') -- (num8);
    \draw(num8') -- (num7);    
    \draw(num10') -- (num5);    
    \draw(num11') -- (num14);
    \draw(num12') -- (num4);
     \draw(num14') -- (num2);

    \draw[dashed](6.3,0.3)--(6.3,-4.3);
\end{tikzpicture}}$$  

By Lemma \ref{lambda-w}, we can find that $$w_\lambda=(-1,\dots,-(n-i),n,-(n-i+1),\dots,-(n-1)).$$  By using the RS algorithm, we have
\[P(^{-}w_{\lambda})=\tiny{\begin{tikzpicture}[scale=0.8,baseline=-50pt]
			\hobox{0}{0}{-n}
			\hobox{0}{1}{-n+2}
			\hobox{0}{2}{\vdots}
                \hobox{0}{3}{n-1}
			\hobox{1}{0}{-n+1}
			\hobox{2}{0}{n}
   \end{tikzpicture}} .\]

If  $-\lam_{2}\geq \lam_1$,  we will have:
$${\begin{tikzpicture}
     \draw(0,-4) node{$ \mu^-: $};
      \draw (0,0) node{$ \lambda^-: $};
    \node (num1) at (0.7,0) {$\lam_{1}$};
    \node (num2) at (1.7,0) {$\lam_{2}$};
    \node (num3) at (2.7,0) {$\dots$};
      \node (num4) at (3.7,0) {$\lam_{n}$};
       \node (num5) at (4.7,0) {$-\lam_{n}$};
    \node (num6) at (5.7,0) {$\dots$};
       \node (num7) at (6.7,0) {$-\lam_{2}$};
      \node (num8) at (7.7,0) {$-\lam_{1}$};

        \node (num1') at (0.7,-4) {$\lam_{1}$};
    \node (num2') at (1.7,-4) {$-\lam_{2}$};
    \node (num3') at (2.7,-4) {$\dots$};
    \node (num4') at (3.7,-4) {$-\lam_{n}$};
       \node (num5') at (4.7,-4) {$\lam_{n}$};
       \node (num6') at (5.7,-4) {$\dots$};
    \node (num7') at (6.7,-4) {$\lam_{2}$};
    \node (num8') at (7.7,-4) {$-\lam_{1}$};

   \draw(num1') -- (num1);
    \draw(num2') -- (num7);
    \draw(num3') -- (num6);
    \draw (num4') -- (num5);
   \draw(num5') -- (num4);
    \draw(num6') -- (num3);    
    \draw(num7') -- (num2);    
    \draw(num8') -- (num8);

    \draw[dashed](4.2,0.3)--(4.2,-4.3);
\end{tikzpicture}}$$  

 By Lemma \ref{lambda-w}, we can find that $$w_\lambda=(-1,\dots,-(n-1),n).$$  In this case, $P(^{-}w_{\lambda})$ is the same as the previous one.

\end{proof}

By using Theorem \ref{hermitian av} or Garfinkle's domino tableau algorithm \cite{Garfinkle1}, we can see that these $2n-1$ $w_{\lam}$'s in Theorem \ref{b-cell-3} belong to the same KL right cell. By Proposition \ref{BCD}, no $L_w$ can have associated variety $\overline{\mathcal{O}}_1$. 

\begin{Cor}\label{b-cell}
  When $G=SO(2,2n-1)$, let $L(\lambda)$  be an integral highest weight Harish-Chandra module. 
 Then we have $V(L(\lambda))=V(L_{w})=\overline{\mathcal{O}}_{2}$ if and only if
	\begin{equation*}
	w_{\lambda}\stackrel{R}{\sim} w\stackrel{R}{\sim} (n,-1,\dots, -(n-2),-(n-1 )).
	\end{equation*}
 And $V(L(\lambda))=V(L_{w_{\lambda}})=\overline{\mathcal{O}}_{0}$ if and only if $w_{\lambda}=-Id$.
	
\end{Cor}

\begin{Cor}\label{size-b}
We use $\mathcal{C}_k$ to denote a Harish-Chandra cell consisting of  highest weight Harish-Chandra $SO(2,2n-1)$-modules $L_w$ with associated variety $\overline{\mathcal{O}}_{k}$. Then we have
$
\#\mathcal{C}_2=2n-1$ and $\#\mathcal{C}_0=1$.
\end{Cor}
\begin{proof}
    From $|\mathcal{W}|=|W_n/W_{n-1}|=2n=\#\mathcal{C}_0+\#\mathcal{C}_2$ and $\#\mathcal{C}_0=1$, we can get $\#\mathcal{C}_2=2n-1$ since there is only one KL cell such that $V(L_w)=\overline{\mathcal{O}}_{2}$ by Corollary \ref{b-cell}.
\end{proof}

Note that the $2n-1$ elements $L_w$ in $\mathcal{C}_2$ correspond to the $2n-1$ $w_{\lam}$'s in Theorem \ref{b-cell-3}.

When $\lam$ is a half integral  highest weight Harish-Chandra module of $SO(2,2n-1)$, the integral root system will be isomorphic to the root system of type $A_1\times B_{n-1}$ and the integral Weyl group $W_{[\lam]}$ will be isomorphic to $S_2\times W_{n-1}$.
From Proposition \ref{BCD} and Theorem \ref{b-cell-3}, we have the following result.

\begin{Cor}\label{b-half}
Suppose $\frg_{\mathbb{R}}=\frso(2,2n-1)$ and $L(\lambda)$ is a half integral highest weight Harish-Chandra module with highest weight $\lambda-\rho$. Write $w_{\lam}=w_A\times w_B$.
Then we have $V(L(\lambda))=\overline{\mathcal{O}}_{1}$ if and only if
	 $$w_A=s_1=(2,1),
$$
and $V(L(\lambda))=\overline{\mathcal{O}}_{2}$ if and only if  $$w_A=Id_{A_1}=(1,2).
$$

\end{Cor}
\begin{proof}
 We write $\lam=(\lam_1,\dots,\lam_n)$. From the construction of $L(\lam)$, we have $\lambda_1-\lam_2\in\frac{1}{2}+\mathbb{Z}$, $\lam_i-\lam_{i+1}\in\mathbb{Z}_{>0}$ for $2\leq i\leq n-1$ and $2\lambda_k\in \mathbb{Z}_{>0}$ for $2\leq k\leq n$ since $\lam$ is half integral. Thus we have two maximal subsequences $x=(\lam_1)$ and $y=(\lam_2,\dots,\lam_n)$. The corresponding   integral root system  $\Phi_{[\lam]}$ is isomorphic to the root system of type $A_1\times B_{n-1}$ and $w_{\lam}=w_{x}\times w_y=w_x\times (-Id)$.

  By Proposition \ref{BCD}, we have $V(L(\lambda))=\overline{\mathcal{O}}_{1}$ if and only if $x=\lam_1>0$, equivalently $w_A=s_1=(2,1)$. And $V(L(\lambda))=\overline{\mathcal{O}}_{2}$ if and only if $x=\lam_1\leq 0$, equivalently $w_A=Id_{A_1}=(1,2)$.

\end{proof}

\subsection{The case of $\mathfrak{sp}(n,\mathbb{R})$}
In this subsection, we assume $\mathfrak{g}=\mathfrak{sp}(n,\mathbb{C})$ and $\mathfrak{g}_{\mathbb{R}}=\mathfrak{sp}(n,\mathbb{R})$.
 From Enright--Howe--Wallach \cite{EHW}, we know that the highest weight module $L(\lambda)$ is an integral  Harish-Chandra module if and only if  $\lam_i-\lam_{i+1}\in \mathbb{Z}_{> 0}$ for $1\leq i\leq n-1$ and $\lam_k\in \mathbb{Z}$ for $1\leq k\leq n$. Therefore, $L_w$ is a highest weight Harish-Chandra $Sp(n,\mathbb{R})$-module if and only if $-w\rho=(\lam_1,\lam_2,\cdots,\lam_n)$ with $\lam_i-\lam_{i+1}\in \mathbb{Z}_{ >0}$ for $1\leq i\leq n-1$. The number of these modules is $|\mathcal{W}|=|W_n/S_{n}|=\frac{n!2^n}{n!}=2^n$.

First we consider the Wallach representations.

\begin{Thm}\label{c-cell-even}
Suppose $\frg_{\mathbb{R}}=\frsp(n,\mathbb{R})$ and $L(\lambda)$ is a highest weight Harish-Chandra module with highest weight $\lambda-\rho =-k c\zeta $. Suppose $w_{\lambda}\in W$ is the minimal length element such that $\mu=w_{\lambda}^{-1}\lambda$ is antidominant, then we have
$$w_\lambda=(\frac{k}{2},-\frac{k}{2}-1,\frac{k}{2}-1,-\frac{k}{2}-2,\frac{k}{2}-2,\dots, -(k-1),1,-k,-(k+1),\dots, -n) 
$$
if $k\in 2\mathbb{Z},1\leq k\leq n.$

\end{Thm}
\begin{proof}
    From \cite{EHW}, we have $\lambda=-k c\zeta +\rho=(n-\frac{k}{2},(n-1)-\frac{k}{2},\dots, 1-\frac{k}{2})$.
By Proposition \ref{BCD}, we have $V(L(\lambda))=\overline{\mathcal{O}}_{k}.$ By Proposition \ref{anti},  we have $\mu=(-(n-\frac{k}{2}),\dots,-\frac{k}{2},-(\frac{k}{2}-1),-(\frac{k}{2}-1),\dots,-1, -1,0)$. Thus we have:
 $$\tiny{\begin{tikzpicture}
     \draw(-1,-4) node{$ \mu^-: $};
      \draw (-1,0) node{$ \lambda^-: $};
    \node (num1) at (0,0) {$n-\frac{k}{2}$};
    \node (num2) at (1-0.1,0) {$\dots$};
    \node (num3) at (2-0.5,0) {$1$};
    \node (num4) at (2.5,0) {$0$};
    \node (num5) at (3,0) {$-1$};
    \node (num6) at (3.6,0) {$\dots$};
    \node (num7) at (4.5,0) {$1-\frac{k}{2}$};
    \node (num8) at (5.4,0) {$\frac{k}{2}-1$};
    \node (num9) at (6.4,0) {$\dots$};
    \node (num10) at (7.4,0) {$1$};
    \node (num11) at (8.3,0) {$0$};
    \node (num12) at (9.2,0) {$-1$};
    \node (num13) at (10.1,0) {$\dots$};
    \node (num14) at (11,0) {$-(n-\frac{k}{2})$};

    \node (num1') at (0,-4) {$-(n-\frac{k}{2})$};
    \node (num2') at (1-0.1,-4) {$\dots$};
    \node (num3') at (2-0.5,-4) {$1-\frac{k}{2}$};
    \node (num4') at (2.5,-4) {$\dots$};
    \node (num5') at (3,-4) {$-1$};
    \node (num6') at (3.6,-4) {$-1$};
    \node (num7') at (4.5,-4) {$0$};
    \node (num8') at (5.4,-4) {$0$};
    \node (num9') at (6.4,-4) {$1$};
    \node (num10') at (7.4,-4) {$1$};
    \node (num11') at (8.3,-4) {$\dots$};
    \node (num12') at (9.2,-4) {$\frac{k}{2}-1$};
    \node (num13') at (10.1,-4) {$\dots$};
    \node (num14') at (11,-4) {$n-\frac{k}{2}$};
    
    \draw(num1') -- (num14);
    \draw(num3') -- (num7);
   \draw(num5') -- (num5);
    \draw (num6') -- (num12);
   \draw(num7') -- (num4);
   \draw(num8') -- (num11); 
   \draw(num9') -- (num3); 
   \draw(num10') -- (num10); 
   \draw(num12') -- (num8);    
   \draw(num14') -- (num1);

    \draw[dashed](5,0.3)--(5,-4.3);
\end{tikzpicture}}$$
By Lemma \ref{lambda-w}, we can find that $$w_\lambda=(\frac{k}{2},-\frac{k}{2}-1,\frac{k}{2}-1,-\frac{k}{2}-2,\frac{k}{2}-2,\dots, -(k-1),1,-k,-(k+1),\dots, -n).$$  
By using the RS algorithm, we have
\[P(^{-}w_{\lambda})=\tiny{\begin{tikzpicture}[scale=0.8,baseline=-70pt]
			\hobox{0}{0}{-n}
			\hobox{0}{1}{-n+1}
			\hobox{0}{2}{\vdots}
                \hobox{0}{3}{-1}
               \hobox{0}{4}{k+1}
                \hobox{0}{5}{\vdots}
                \hobox{0}{6}{n}
			\hobox{1}{0}{1}
			\hobox{1}{1}{\vdots}
                 \hobox{1}{2}{k-1}
   \end{tikzpicture}} .\]

\end{proof}

Note that when $k=2$, $w_\lambda=(\frac{k}{2},-(\frac{k}{2}+1),\dots ,-(k-2),2,-(k-1),1,-k,\dots, -n)$ will be reduced to $(1,-2,-3,\dots,-n)$. When $k=4$, $w_\lambda$ will be reduced to $(2,-3,1,-4,-5,\dots,-n)$. When $k=n=2m$, $w_\lambda$ will be reduced to $(\frac{k}{2},-\frac{k}{2}-1,\frac{k}{2}-1,-\frac{k}{2}-2,\frac{k}{2}-2,\dots,-\frac{k}{2}-(\frac{k}{2}-1),\frac{k}{2}-(\frac{k}{2}-1),-n)=(\frac{k}{2},-\frac{k}{2}-1,\frac{k}{2}-1,-\frac{k}{2}-2,\frac{k}{2}-2,\dots,-n+1,1,-n)$.

\begin{Thm}\label{c-cell-odd}
Suppose $\frg_{\mathbb{R}}=\frsp(n,\mathbb{R})$ and $L(\lambda)$ is an integral highest weight Harish-Chandra module with highest weight $\lambda-\rho $ and $V(L(\lambda))=\overline{\mathcal{O}}_{n}$ with $n$ being odd. Suppose $w_{\lambda}\in W$ is the minimal length element such that $\mu=w_{\lambda}^{-1}\lambda$ is antidominant, then we have
$$w_\lambda\stackrel{R}{\sim}(n,n-1,\dots,2,1). 
$$


\end{Thm}
\begin{proof}
    Set $w=(n,n-1,\dots,2,1)$. Then $-w\rho=(-1,-2,\dots,n)$ and $V(L_w)=\overline{\mathcal{O}}_{n}$ by Proposition \ref{BCD}.

    For an integral highest weight Harish-Chandra module $L(\lam)$, if $V(L(\lambda))=V(L(w_{\lambda}))=\overline{\mathcal{O}}_{n}=V(L_w)$, we will have $w_{\lam}\stackrel{R}{\sim}w$ by Theorem \ref{hermitian av}.


\end{proof}

By using Theorem \ref{hermitian av}, we have the following result.

\begin{Cor}\label{c-cell}
  When $G=Sp(n,\mathbb{R})$, let $L(\lambda)$  be an integral highest weight Harish-Chandra module. 
 Then we have the followings:
 \begin{enumerate}
     \item suppose $k\in 2\mathbb{Z}$ and $1\leq k\leq n$, then $V(L(\lambda))=V(L_{w})=\overline{\mathcal{O}}_{k}$   if and only if
	$$w_{\lambda}\stackrel{R}{\sim}
w\stackrel{R}{\sim}(\frac{k}{2},-\frac{k}{2}-1,\frac{k}{2}-1,-\frac{k}{2}-2,\frac{k}{2}-2,\dots, -(k-1),1,-k,-(k+1),\dots, -n).
$$
     \item suppose $n$ is odd, then
 $V(L(\lambda))=V(L_w)=\overline{\mathcal{O}}_{n}$ if and only if 
$$w_\lambda\stackrel{R}{\sim}w\stackrel{R}{\sim}(n,n-1,\dots,2,1). 
$$
 \end{enumerate}


\end{Cor}

\begin{Cor}\label{size-c}
We use $\mathcal{C}_k$ to denote a Harish-Chandra cell consisting of  highest weight Harish-Chandra $Sp(n,\mathbb{R})$-modules $L_w$ with associated variety $\overline{\mathcal{O}}_{k}$ for some $k\in 2\mathbb{Z}$ and $0\leq k\leq n$. Then we have
$$
\#\mathcal{C}_k=\begin{cases}
   \binom{n+1}{\frac{k}{2}}, &\text{~if~} k\in 2\mathbb{Z} \text{~and~} 0\leq k\leq n,\\
    \binom{n}{\frac{n-1}{2}}, &\text{~if~} k=n \text{~is~odd}.
\end{cases}
$$

\end{Cor}
\begin{proof}
    This follows from \cite[Prop. 33]{BZ}.
\end{proof}

Note that the elements in $\mathcal{W}$ are described in \cite[\S 3]{BZ}.

When $\lam$ is a half integral  highest weight Harish-Chandra module of $Sp(n,\mathbb{R})$, the integral root system will be isomorphic to the root system of type $D_n$ and the integral Weyl group $W_{[\lam]}$ will be isomorphic to $W'_n$.
From Proposition \ref{BCD} and Theorem \ref{d-wlam}, we have the following result.

\begin{Cor}\label{c-half}
Suppose $\frg_{\mathbb{R}}=\frsp(n,\mathbb{R})$ and $L(\lambda)$ is a half integral highest weight Harish-Chandra module with highest weight $\lambda-\rho$. Then we have $V(L(\lambda))=\overline{\mathcal{O}}_{2k+1}$ for some $k\in \mathbb{Z}$ and $1\leq 2k+1\leq n$ if and only if
	 $$w_\lambda\stackrel{LR}{\sim}z_k:=\begin{cases}
    (k+1,-(k+2),k, -(k+3),k-1,\dots,&\\
    -2k,2,-(2k+1),1,-(2k+2),\dots, -n), &\text{~if~} n-k \text{~is~ odd},\\
     (k+1,-(k+2),k, -(k+3),k-1,\dots,&\\
     -2k,2,-(2k+1),-1,-(2k+2),\dots, -n), &\text{~if~} n-k \text{~is~ even}.
\end{cases}
$$
\end{Cor}
\begin{proof}
  When $w_{\lam}$ belongs to the given two-sided cell containing $z_k$,  we have  $q_2^{\ev}=k$ since there are $k$ even boxes in the second columns of the Young tableaux appeared in the argument of Theorem \ref{d-wlam}. By Proposition \ref{BCD}, we have $V(L(\lambda))=\overline{\mathcal{O}}_{2k+1}$.

    Conversely, suppose that $L(\lam)$ is a half integral highest weight Harish-Chandra module and we have $V(L(\lambda))=\overline{\mathcal{O}}_{2k+1}$.
    Then by Proposition \ref{BCD}, we have $k=q_2^{\ev}$ and $w_\lambda\stackrel{LR}{\sim}z_k$. 
\end{proof}

\subsection*{Acknowledgments}
 
Z. Bai was supported    by the National Natural Science Foundation of China  (Grant No. 12171344). Y. Bao was supported    by the National Natural Science Foundation of China  (Grant No. 11971315 and No. 11801117). X. Xie was supported  by the National Natural Science Foundation of China (Grant No. 12171030 and No. 12431002). We thank the referee for his/her helpful comments and suggestions.

\printbibliography

\end{document}